\documentclass[11pt, oneside]{article}   	
\usepackage{geometry}                		
\usepackage{graphicx}				
\usepackage{amssymb}
\usepackage{amsmath}
\usepackage{amsthm}
\usepackage[autostyle, english = american]{csquotes}
\usepackage{fncylab}
\usepackage{environ}
\usepackage{tikz}
\usepackage{pgfplots}
\pgfplotsset{samples=200}
\usetikzlibrary{external}

\newtheorem{definition}{Definition}[section]
\newtheorem{theorem}{Theorem}[section]
\newtheorem*{pretheorem}{Theorem}

\newtheorem{lemma}{Lemma}[section]

\newtheorem{conditions}{Conditions}[section]

\labelformat{figure}{Figure #1}
\labelformat{equation}{(#1)}
\labelformat{section}{Section #1}
\labelformat{subsection}{Section #1}
\labelformat{table}{Table #1}
\labelformat{definition}{Definition #1}
\labelformat{theorem}{Theorem #1}
\labelformat{conjecture}{Conjecture #1}
\labelformat{lemma}{Lemma #1}
\labelformat{conj}{Conjecture #1}
\labelformat{corollary}{Corollary #1}
\labelformat{remark}{Remark #1}
\labelformat{proposition}{Proposition #1}
\labelformat{conditions}{#1}
\labelformat{item}{#1}

\newcounter{asyfigcntr}
  \NewEnviron{asy}{%
    \stepcounter{asyfigcntr}%
    \includegraphics{\jobname-\theasyfigcntr}%
  }

\newcommand{\figscrew}{
\begin{figure}
    \label{fig:screw}
    \center \includegraphics[width=0.4\linewidth]{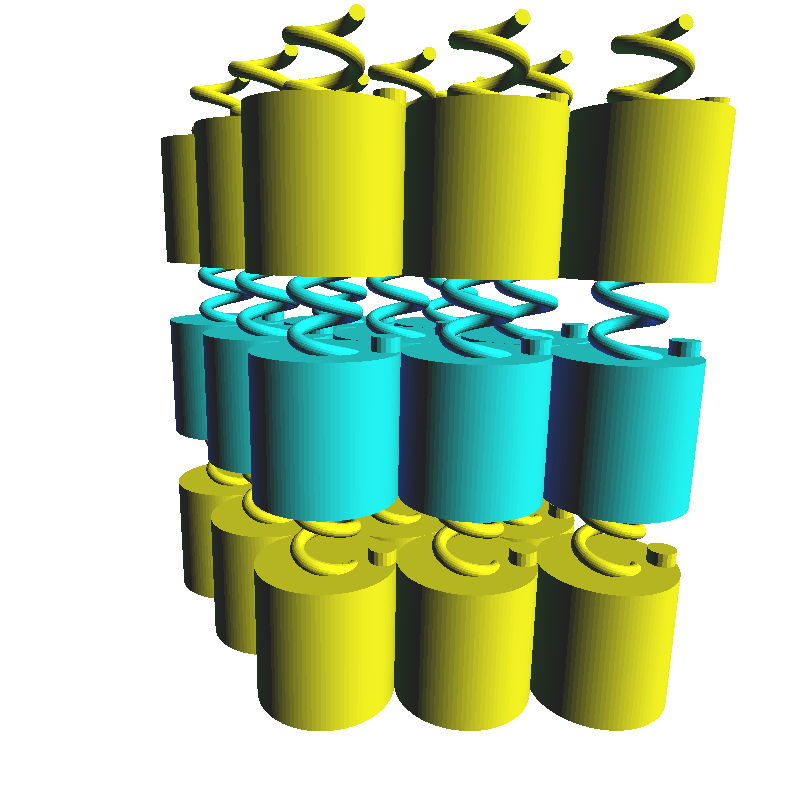}
    \caption{Sample of the packing arrangement of cylinders with screws and screw holes described in the text.
	The arrangement can be condensed, but in a way that is not continuous in the topology given by $\delta_R$.}
\end{figure}
}

\let\origthanks\thanks
\renewcommand\thanks[1]{\begingroup\let\rlap\relax\origthanks{#1}\endgroup}

\title{The local optimality of the double lattice packing}
\author{Yoav Kallus\thanks{Santa Fe Institute, 1399 Hyde Park Road, Santa Fe, New Mexico, USA} and W\"oden Kusner\thanks{Graz University of Technology, Institute for Analysis and Computational Number Theory, Steyrergasse 30/II, 8010 Graz, Austria}}

\begin{document}
\bibliographystyle{alpha}
\maketitle
\abstract{
This paper introduces a technique for proving the local optimality of packing configurations.  Applying this technique to a general convex polygon, we prove that the construction of the optimal double lattice packing by Kuperberg and Kuperberg is also locally optimal in the full space of packings.
}

\section{Introduction}

This paper began as an investigation of the optimality of the double lattice packing for pentagons and heptagons.  In \cite{kuperberg1990double}, Kuperberg and
Kuperberg describe a recipe for finding the double lattice packing of congruent planar convex bodies with maximal density by solving an optimization
problem over inscribed parallelograms. This optimization problem is usually tractable, and in the case of convex polygons can be solved by an algorithm with
running time linear in the number of vertices \cite{mount1991densest}. As examples, Kuperberg and Kuperberg construct the densest double lattice packing for
both regular pentagons and regular heptagons and show that these packings have densities of $(5-\sqrt{5})/3 = 0.92131\dots$ and $0.8926\dots$ respectively.
These are the current records and possibly the best \textit{general} packings of the plane by regular pentagons and by regular heptagons.

Starting around the turn of the century, a significant push, both theoretical and computational, arose to answer some of the most basic yet frustrating questions in
the theory of packing problems (for background on packing problems, see \cite{brass2005research}, \cite{conway1999recent}, \cite{groemer1963existenzsatze}).
Along with the proof and formal verification of the Kepler conjecture \cite{hales2015formal}, a number of other results 
on sphere packing in higher dimensions have proved illuminating \cite{cohn2003new,vance2011improved,venkatesh2013note}.
For packings by congruent anisotropic bodies, sharp results are limited mostly to the plane, where the best packings of all centrally-symmetric bodies are
achieved by lattices \cite{fejes1950some}, and a series of sparse results in higher dimensions \cite{bezdek2013dense}.

Among general convex bodies, the problem of finding the best packing of regular pentagons serves as a toy model for harder problems,
like finding the best packing of regular tetrahedra. However, the pentagon problem is still not a tractable one. Explicit upper bounds
for the packing density of regular tetrahedra and octahedra are better than the trivial unity upper bound by minuscule margins \cite{gravel2011upper}.
A semidefinite programming (SDP) approach has been suggested by Oliveira and Vallentin to calculate improved upper bounds \cite{mario2013computing}.
Though the SDP method has not yet yielded a nontrivial upper bound for packing of tetrahedra, it has been used to obtain an upper bound of $0.98103$ on
the density of regular pentagon packings. There remains a sizable gap between the highest density achieved for pentagon packings and this upper bound.

A long-open problem, still wide open even in the plane, asks for the pessimal convex body for packing, that is,
the shape that has the lowest maximum packing density \cite{brass2005research,bezdek2013dense,kallus2014ball}.
In the class of centrally-symmetric bodies in the plane, it is Reinhardt \cite{reinhardt1933dichteste} who conjectured that a smoothed octagon is the minimizer.
In the class of general convex bodies in the plane, it is conjectured to be the regular heptagon \cite{kallus2015pessimal}.
However, even though the maximum-density double lattice is conjectured to achieve the 
maximum packing density for the regular heptagon, no sharp upper bound has been proved.

The regular pentagon and heptagon are cases of special interest, and we initially sought out to investigate whether their optimal double-lattice packing can be shown
to be also optimal among a broader class of packings. We were able to show that these packings are optimal at least in some neighborhood in the space of all packings.
Furthermore, we discovered that our method can be generalized to all convex polygons. We demonstrate that, while double lattices are in general not globally
optimal, they are always at least locally optimal.

\begin{pretheorem}
If a double lattice packing is an isolated local maximum for density among double lattices and is not one of a two exceptional cases, then it is a local maximum for density among all packings.
\end{pretheorem}

The precise meanings of the terms used will be elucidated in the rest of the paper.

\section{Theoretical Preparations}
\subsection{Local Optimality}

We will look at packings of congruent copies of a body $K$. That is, every
element of the packing is given by $\xi(K)$, where $\xi\in E(n)$ is an isometry of Euclidean
space. It will be convenient to assume that the reference body $K$ is situated so that
its interior contains the origin. The isometry group $E(n)$ of $\mathbb{R}^n$ can be considered
as a subgroup of $SL_{n+1}(\mathbb{R})$ that preserves the plane $(x_1,\ldots,x_n,1)\in\mathbb{R}^{n+1}$.
This identification gives us the Frobenius norm $N(\xi) = \|\xi - \mathrm{Id}\|$ for $\xi\in E(n)$. Some useful
inequalities for this norm that we will use are
\begin{equation}\label{eq:norm-ineq}
    \begin{aligned}
	N(\xi^{-1}) &\le \|\xi^{-1}\|N(\xi)\\
	N(\xi\xi') &\le N(\xi)+N(\xi')+N(\xi)N(\xi')\\
	N(\psi \xi \psi^{-1}) &\le \|\psi\|\|\psi^{-1}\|N(\xi)\\
	\|\xi(0)\| \le N(\xi) &\le \|\xi(0)\|+2n^{1/2}\text.
    \end{aligned}
\end{equation}

\begin{definition}
    Let $\Xi$ be a set of isometries. The limit
    \begin{equation}
	d(\Xi)=\lim_{t\to\infty} \frac{\mathrm{vol}\,B(0,t)}{|\{\xi\in\Xi : \xi(0)\in B(0,t)\}|}\text,
    \end{equation}
    if it exists, is its \textit{mean volume}.  
    The limits superior and inferior are its upper and lower mean volumes, denoted $\overline{d}(\Xi)$ and $\underline{d}(\Xi)$.
    We say $\Xi$ is a $(r,R)$-set if the point set $\{\xi(0):\xi\in\Xi\}$ has a packing
    radius at least $r$ and a covering radius at most $R$.
\end{definition}

\begin{definition}
    Let $K$ be a compact set with interior. We say that $\Xi$ is \textit{admissible} for $K$ if
    the interiors of $\xi(K)$ and $\xi'(K)$ are disjoint for any two distinct isometries $\xi,\xi'\in\Xi$.
    We say furthermore that $\Xi$ is \textit{saturated} if there is no $\xi \in E(n)\setminus\Xi$ such that $\Xi\cup\{\xi\}$
    is again admissible.
\end{definition}

There are $r(K)$ and $R(K)$ such that when $\Xi$ is admissible and saturated, then $\Xi$ is
a $(r(K),R(K))$-set. As a consequence, such sets are countable.

\begin{definition}
    Given two $(r',R')$-sets $\Xi$ and $\Xi'$ of isometries, we define the premetric 
    \begin{equation}
	\begin{aligned}
	    \delta_R(\Xi,\Xi') = \inf_\text{bij.} \sup \{&N(\xi^{-1}\psi(\psi')^{-1}\xi'):\\ & \xi,\psi\in\Xi \text{ such that } \|\xi(0)-\psi(0)\|\le 2R \text{ or } \|\xi'(0)-\psi'(0)\|\le 2R\}\text.
	\end{aligned}
    \end{equation}
    The infimum is over all bijections $(\cdot)'\colon\Xi\to\Xi'$.
\end{definition}

When $R\ge R'$, $\delta_R(\Xi,\Xi')=0$ if and only if $\xi = \psi \xi'$ for some $\psi\in E(n)$ and some
bijection. Consider a body $K$. When $R\ge R(K)$, $\delta_R(\Xi,\Xi')$ induces a metric on the space of admissible
$(r,R)$-sets up to overall isometry, which includes the saturated sets as a subset.

\begin{definition}
    We say an admissible and saturated set $\Xi$ is \textit{strongly extreme} for $K$ if 
    there is $R>0$ and $\epsilon>0$, such that whenever $\delta_R(\Xi,\Xi')<\epsilon$, then
    either $\Xi'$ is inadmissible or $\underline{d}(\Xi')\ge \overline{d}(\Xi)$. 
\end{definition}

We stop to discuss why we define the topology in the space of packing arrangements in the way that we do.
A naive choice of topology is the one given by the metric
\begin{equation}
    \delta_H(\Xi,\Xi') = \inf_{\text{bij.},\psi\in E(n)}\sup_{\xi\in\Xi} N(\xi^{-1}\psi\xi')\text,
\end{equation}
which we call the Hausdorff metric. However, under the Hausdorff topology
the mean volume is locally constant, and any packing is trivially locally optimal. For the topology to allow
the mean volume to vary locally, it must allow elements that are increasingly far
apart to move by the action of increasingly different isometries. One metric satisfying this criterion,
similar to the one used in Ref.\ \cite{Kuperberg2000notions}, is
\begin{equation}
    \inf_{\text{bij.},\psi\in E(n)} \sup_{\xi\in\Xi} \min(N(\xi^{-1}\psi\xi'),1/\min(\|\xi(0)\|,\|\xi'(0)\|))\text.
\end{equation}
While suited for the discussions of recurrence
there, here this metric will yield the result that any packing that is not as dense as
the densest packing of $K$ is not even locally optimal, since the density can be improved by a finite amount
at arbitrarily small distance by completely changing the packing outside a ball of arbitrarily large radius.
Under our definition, for a family of packings to have a reference packing as a limit, it must be
the case that in every bounded region of space, there is a packing in the family that agrees with the
reference packing to an arbitrarily small Hausdorff distance.

\figscrew

Importantly, there are deformations that intuitively feel continuous but are not continuous under our topology. 
One example is constructed by decorating a cylinder
with a screw on its top base and a corresponding screw hole boring into its bottom base. Consider the packing
where each cylinder is screwed into a cylinder above it, in such a way that the two cylinders
are related to each other by a translation, and the screw is not completely screwed in. This
creates a column of cylinders, copies of which we arrange in a triangular grid.
Since every screw is not completely screwed in, the density of the packing can be increased by screwing each screw in
further. Since the interlayer spacing is related to the relative rotation between
cylinders on the two layers, even an arbitrarily small consistent decrease in interlayer spacing
will cause some layers to be rotated by at least some finite angle. Because the orientation of triangular
grid remains unchanged, this means that there is always a cylinder whose arrangement of neighbors, from the frame of reference
in which it remains fixed, has changes by a finite extent.
Therefore, this motion is not continuous in the topology we defined. As we prove some packing arrangements strongly extreme,
it is worth keeping in mind what kinds of local improvement such a result rules out and what kinds
are not ruled out.

Nevertheless, compared to some previously introduced notions of local optimality, our notion of strong extremality is both stronger
and more widely applicable. The notions of an \textit{extreme}
lattice packing \cite{martinet2003perfect} and a \textit{periodic-extreme} periodic packing \cite{schurmann2013strict} apply only to special
classes of packings. We show that strong extremality, which applies more generally,
implies extremality and periodic-extremality in these special classes.

\begin{definition}
    A set of isometries $\Lambda$ is called a (full rank) lattice if it is an $(r,R)$-set for some $r>0$ and $R<\infty$, it consists
    only of translations, and it is closed under composition and inversion.
\end{definition}
\begin{definition}
    A lattice $\Lambda$ is \textit{extreme} for a compact set $K$ if it is admissible
    for $K$ and there exists $\epsilon>0$ such that for all $T\in GL_n(\mathbb{R})$, either $T[\Lambda]=\{T[\lambda]=T\lambda T^{-1}:\lambda\in\Lambda\}$
    is inadmissible for $K$, $\det T\ge1$, or $\|T-\mathrm{Id}\|>\epsilon$.
\end{definition}
\begin{theorem}
    If a lattice $\Lambda$ is strongly extreme for $K$, then $\Lambda$ is extreme for $K$.
\end{theorem}
\begin{proof}
    If $\Lambda$ is not extreme for $K$, then there for all $\epsilon>0$, there exists $T\in GL_n(\mathbb{R})$ such that
    $\|T-\mathrm{Id}\|\le\epsilon$, $\det T<1$, and $T[\Lambda]$ is admissible for $K$.
    We have $N(\lambda_1^{-1}\lambda_2\allowbreak (T[\lambda_2])^{-1}T[\lambda_1]) = N((T-\mathrm{Id})[\lambda_1\lambda_2^{-1}])
    \le\epsilon\|\lambda_1(0)-\lambda_2(0)\|$. 
    Thus, $\delta_R(\Lambda,T\Lambda)<2R\epsilon$, and for arbitrarily small $\epsilon$, $T\Lambda$
    is an admissible packing of lower mean volume in an arbitrary neighborhood of $\Lambda$. Therefore,
    $\Lambda$ is not strongly extreme.
\end{proof}

\begin{definition}
    A set $\Xi\subset E(n)$ is periodic if it is of the form $\Lambda\Psi=\{\lambda\psi :\lambda\in\Lambda,\psi\in\Psi\}$,
    where $\Lambda$ is a lattice and $\Psi$ is finite.
\end{definition}
\begin{definition}\label{def:perext}
    A periodic set $\Xi=\Lambda\Psi$ is periodic-extreme for $K$ if it is admissible for $K$ and
    whenever $\tilde\Lambda\subseteq\Lambda$ is a sublattice of $\Lambda$ and
    $\tilde\Psi$ is a set of $|\Lambda/\tilde\Lambda|$ translations such that $\Lambda = \tilde\Lambda\tilde\Psi$, there exists
    $\epsilon>0$ such that for all $T\in GL_n(\mathbf{R})$ and $\phi\colon\tilde\Psi\times\Psi\to E(n)$ we have either
    $\{T[\tilde\lambda]\tilde\psi\psi\phi(\tilde\psi,\psi):\tilde\lambda\in\tilde\Lambda,\tilde\psi\in\tilde\Psi,\psi\in\Psi\}$ is inadmissible for $K$,
    $\det T\ge 1$, $\|T-\mathrm{Id}\|>\epsilon$, or $\|\phi(\tilde\psi,\psi)-\mathrm{Id}\|>\epsilon$ for
    some $\tilde\psi\in\tilde\Psi,\psi\in\Psi$.
\end{definition}
\begin{theorem}
    If a periodic set $\Xi$ is strongly extreme for $K$, then it is periodic-extreme for $K$.
\end{theorem}
\begin{proof}
    If $\Xi$ is not periodic-extreme, then there exists $\tilde\Lambda\subseteq\Lambda$ and $\tilde\Psi$ as in
    \ref{def:perext}, such that for all $\epsilon>0$ there exists $T\in GL_n(\mathbb{R})$ and
    $\phi:\tilde\Psi\times\Psi\to E(n)$ such that
    $\Xi'=\{T[\tilde\lambda]\tilde\psi\psi\phi(\tilde\psi,\psi):\tilde\lambda\in\tilde\Lambda,\tilde\psi\in\tilde\Psi,\psi\in\Psi\}$ is admissible for $K$,
    $\det T < 1$, $\|T-\mathrm{Id}\|\le\epsilon$, and $\|\phi(\tilde\psi,\psi)-\mathrm{Id}\|\le\epsilon$ for
    all $\tilde\psi\in\tilde\Psi,\psi\in\Psi$.
    Since $\Psi\times\Psi'$ is finite, we have some bound $\|\psi\tilde\psi\|,\|\tilde\psi^{-1}\psi^{-1}\|<M$.
    Consider the elements $\xi_1=\tilde\lambda_1\tilde\psi_1\psi_1$, $\xi_2=\tilde\lambda_2\tilde\psi_2\psi_2$,
    $\xi_1,\xi_2\in\Xi$ and the corresponding elements $\xi_1'=T[\tilde\lambda_1]\tilde\psi_1\psi_1\phi(\tilde\psi_1,\psi_1)$
    $\xi_2'=T[\tilde\lambda_2]\tilde\psi_2\psi_2\phi(\tilde\psi_2,\psi_2)$, $\xi_1',\xi_2'\in\Xi'$.
    Using the inequalities \ref{eq:norm-ineq}, it is fairly straightforward to see that
    $\|\xi_1^{-1}\xi_2(\xi_2')^{-1}\xi_1'\|\le C\epsilon$, where $C$ depends on $M$ and $R$.
    Therefore $\delta_R(\Xi,\Xi')$ can be made arbitrarily small, $\Xi'$ is admissible,
    and $d(\Xi') = \det T d(\Xi)<d(\Xi)$, so $\Xi$ is not strongly extreme.
\end{proof}

We now derive a general method for proving strong extremality that we will use in the following sections.

\begin{definition}
    Let $\Xi$ be an $(r,R)$-set of isometries.
    Let $\mathcal T$ be a simplicial complex whose underlying space is $\overline{\mathcal T}=\mathbb{R}^n$, whose vertices
    are in $\{\xi(0):\xi\in\Xi\}$, and whose simplices $s$ have underlying space $\overline s = \mathrm{conv}_{\xi(0)\in s} \xi(0)$
    with diameter uniformly bounded from above and inradius uniformly bounded from below.
    Let $p:\mathcal T_n\to \Xi$ be a labeling of the full-dimensional simplices, such that
    \begin{itemize}
	\item $\xi(0)\in s$ whenever $p(s) =\xi$.
	\item $\mathrm{vol}\,\mathcal P_\xi  = v$ for all $\xi$, where $\mathcal P_\xi =
	    \overline{p^{-1}(\xi)}=\bigcup_{s\text{ s.t. }p(s)=\xi}\overline{s}$.
    \end{itemize}
    Then $(\mathcal T,p)$ is called a honeycomb of $\Xi$, and $\mathcal P_\xi$, $\xi\in\Xi$, are the cells of the honeycomb.
\end{definition}

It is easy to verify that if $\Xi$ has a honeycomb with cells of volume $v$, then its mean volume is $d(\Xi)=v$.
We denote by $\Xi_s,\Xi_\xi\subset\Xi$, the set of elements $\xi\in\Xi$, such that $\xi(0)$ is a vertex
of $s\in\mathcal T$, or respectively a vertex of any simplex in $p^{-1}(\xi)$.
When we consider another set $\Xi'$ in bijection $(\cdot)'\colon\Xi\to\Xi'$ with $\Xi$,
such that $\delta_R(\Xi,\Xi')<r/2$, the triangulation $\mathcal T$ gives us a new triangulation
$\mathcal T' = \{s' = \{\xi'(0): \xi(0)\in s\} : s\in\mathcal T\}$.
The new triangulation gives new cells $\mathcal P'_{\xi} = \bigcup_{s, p(s)=\xi} \overline s'$.

For a specified honeycomb $(\mathcal T,p)$ of $\Xi$, we can consider the finite volume-minimization problem for the individual cells $\mathcal P_\phi$:
\begin{equation}\label{eq:opt-cell}
    \begin{aligned}
	\text{minimize }&\mathrm{vol}\,\mathcal P'_\phi=\sum_{s,p(s) = \phi} \mathrm{vol}\,\mathrm{conv}_{\xi(0)\in s} \xi'(0)\text, \\
	\text{over }&(\cdot)'\colon \Xi_\phi\to E(n)\text,\\
	\text{subj.\ to }&\mathrm{int}\,\xi'_1(K)\cap\mathrm{int}\,\xi'_2(K) = \emptyset \text{ for all }\xi_1,\xi_2\in\Xi_\phi\text,\\
	&N(\xi^{-1}\xi')<\epsilon \text{ for all }\xi\in\Xi_\phi\text.
    \end{aligned}
\end{equation}
Since, for a fixed $\xi\in\Xi$, $\xi'$ may appear in more than one cell-restricted problems,
$\Xi$ might be strongly extreme for $K$ without the restriction of $\mathrm{Id}\colon\Xi\to E(n)$ optimizing any of
these restricted problems. However, it is reasonable to expect, and in fact we prove, that if the restriction of the identity
optimizes all of these problems, then $\Xi$ is strongly extreme.

\begin{theorem}
    \label{thm:noaux}
    Let $\Xi$ be admissible and saturated for $K$ and let $(\mathcal T,p)$ be a honeycomb of $\Xi$.
    If there exists $\epsilon>0$ such that $\xi'=\xi$, $\xi\in \Xi_\phi$, minimizes \ref{eq:opt-cell} for every cell $\mathcal P_\phi$,
    then $\Xi$ is strongly extreme.
\end{theorem}

\begin{proof}
    Consider a set $\Xi'$ admissible for $K$ with $\delta_{M}(\Xi,\Xi')<\epsilon'$,
    where $M$ is a uniform bound on the distance between two vertices in the same cell.
    When $\epsilon'$ is small enough, then $N(\xi^{-1} \phi (\phi')^{-1} \xi') < \epsilon$ for all $\xi\in\Xi_\phi$.
    Let $\xi'' = \phi (\phi')^{-1}\xi'$ for all $\xi\in\Xi_\phi$, then $(\cdot)''$ 
    is feasible for \ref{eq:opt-cell} and  we have
    $\mathrm{vol}\,\mathcal P'_\phi = \mathrm{vol}\,\mathcal P''_\phi \ge \mathrm{vol}\,\mathcal P_\phi = v$.

    We now wish to bound the lower mean volume $\underline{d}(\Xi')$ from below. From the definition,
    $\underline{d}(\Xi') = \lim\inf_{t\to\infty} \mathrm{vol}\,B(0,t) / |\Xi'_t|$,
    where $\Xi'_t = \{\xi\in\Xi:\xi'(0)\in B(0,t)\}$. The limit does not
    change if we replace $t$ in the numerator $\mathrm{vol}\,B(0,t)$ by $t+2M$.
    Since $B(0,t+2M)$ includes all the cells $\mathcal P'_\xi$ for $\xi\in\Xi'_t$,
    the volume of the ball must be at least $v |\Xi'_t|$, and $\underline{d}(\Xi')\ge v = d(\Xi)$.
    Therefore, $\Xi$ is strongly extreme.
\end{proof}

\ref{thm:noaux} will be strong enough to prove that the densest known packing of regular pentagons is strongly extreme.
However, to prove the same for regular heptagons, we will need a stronger version that allows us to
introduce auxiliary objectives. We will consider instead of \ref{eq:opt-cell}, a modified optimization problem:
\begin{equation}\label{eq:opt-aux}
    \begin{aligned}
	\text{minimize }&\mathrm{vol}\,\mathcal P'_\phi + f_\phi( (\xi')_{\xi\in\Xi_\phi} )\text, \\
	\text{over }&(\cdot)'\colon \Xi_\phi\to E(n)\text,\\
	\text{subj.\ to }&\mathrm{int}\,\xi'_1(K)\cap\mathrm{int}\,\xi'_2(K) = \emptyset \text{ for all }\xi_1,\xi_2\in\Xi_\phi\text,\\
	&N(\xi^{-1}\xi')<\epsilon \text{ for all }\xi\in\Xi_\phi\text.
    \end{aligned}
\end{equation}

We say that the set of auxiliary functions is \textit{negligible in the aggregate} (cf. Ref.\ \cite{hales2012blueprint}, p.\ 149)
if there exist $R,\epsilon,C,$ and $T$ such that whenever $d_R(\Xi,\Xi')<\epsilon$,
we have $\sum_{\phi\in\Xi'_t}f_\phi( (\xi')_{\xi\in\Xi_\phi} ) \le C t^{n-1}$ for all $t>T$.
The auxiliary function is isometry-invariant if
$f_\phi( (\psi\xi')_{\xi\in\Xi_\phi} ) = f_\phi( (\xi')_{\xi\in\Xi_\phi} )$ for all $\psi\in E(n)$.
It is straightforward to extend
the proof of \ref{thm:noaux} to obtain

\begin{theorem}
    \label{thm:aux}
    Let $\Xi$ be admissible and saturated for $K$, let $(\mathcal T,p)$ be a honeycomb of $\Xi$,
    and let $f_\phi$, $\phi\in\Xi$, be negligible in the aggregate and isometry-invariant.
    If there exists $\epsilon>0$ such that $\xi'=\xi$, $\xi\in \Xi_\phi$, minimizes \ref{eq:opt-aux} for each cell $\mathcal P_\phi$,
    then $\Xi$ is strongly extreme.
\end{theorem}

\subsection{Double lattices}\label{sec:doublatt}

In this section, we focus on convex bodies in 2 dimensions. We recapitulate some of the theory
of double lattices, due to Kuperberg and Kuperberg \cite{kuperberg1990double} and Mount
\cite{mount1991densest}.

\begin{definition}
    A chord of a convex body $K$ is a line segment whose endpoints lie on the boundary of $K$.
    A chord is an affine diameter if there is no longer chord parallel to it.
\end{definition}

\begin{definition}
    The convex hull of two parallel chords that are half the length of the parallel affine diameter
    is called a half-length parallelogram.
\end{definition}

An affine diameter of $K$ does not in general uniquely determine a half-length parallelogram (e.g., when
a parallel edge of sufficient length exists), but it always uniquely determines its area.

\begin{definition}
    A set $\Lambda\subset E(n)$ is called a (full rank) double lattice if it is an $(r,R)$-set for some
    $r>0$ and $R<\infty$, it consists of translations and point reflections, it is closed under composition
    and inversion, and it is not a lattice (that is, includes at least one point reflection).
\end{definition}

An $n$-dimensional double lattice is generated by a lattice and a point reflection, or alternatively
by reflections about $n+1$ affine-independent points or by reflections about the $2n$ vertices of
a parallelepiped.

\begin{theorem}[Kuperberg and Kuperberg, Mount]\label{thmkk}
    For a planar convex body $K$, an admissible double lattice of smallest mean area 
    is generated by reflection about the vertices of a half-length parallelogram.
\end{theorem}

Kuperberg and Kuperberg use extensive parallelograms, inscribed parallelograms with all edge lengths greater than half the length of the parallel affine diameter.
They later restrict the analysis to the set of half-length parallelograms \cite{kuperberg1990double}.
Mount gives an explicit proof that it suffices to consider only the half-length parallelograms \cite{mount1991densest}.

\begin{figure}
\begin{center}
\begin{asy}
import cseblack;
import olympiad;
usepackage("amssymb");
size(200);
pair K0=D("\mathbf{k}_0,\mathbf{p}_1",(1,0),E);
pair K1=D("\mathbf{k}_1",rotate(360./9.)*K0,NE);
pair K2=D("\mathbf{k}_2",rotate(360./9.)*K1,N);
pair K3=D("\mathbf{k}_3",rotate(360./9.)*K2,NW);
pair K4=D("\mathbf{k}_4",rotate(360./9.)*K3,W);
pair K5=D("\mathbf{k}_5",rotate(360./9.)*K4,W);
pair K6=D("\mathbf{k}_6",rotate(360./9.)*K5,SW);
pair K7=D("\mathbf{k}_7",rotate(360./9.)*K6,S);
pair K8=D("\mathbf{k}_8",rotate(360./9.)*K7,SE);
path KK = D(K0--K1--K2--K3--K4--K5--K6--K7--K8--cycle,linewidth(1));
pair P1=(1., 0.);
pair P2=D("\mathbf{p}_6",(0.3706844903100875, -0.87104878486755),SE);
pair P3=D("\mathbf{p}_5",(-0.5991618200828668, -0.7478489484525013),SW);
pair P4=D("\mathbf{p}_4",(-0.9396926207859084, 0.24639967283009717),NE);
pair P5=D("\mathbf{p}_3",(-0.4058322193092865, 0.882629724233649),N);
pair P6=D("\mathbf{p}_2",(0.5640140910836678, 0.7594298878186004),NE);
path PP = P2--P3--P5--P6--cycle;
fill(PP,lightgray);
D(PP);
D(P1--P4);
pair P4x = P4;

pair P1=(1., 0.);
pair P2=(0.4403241757329789, 0.8308422937423627);
pair P3=(-0.5295221346599754, 0.8308422937423627);
pair P4=(-0.9396926207859084, 0.);
pair P5=(-0.5295221346599754, -0.8308422937423627);
pair P6=(0.4403241757329789, -0.8308422937423627);
path PP = P2--P3--P5--P6--cycle;
D(PP,dashed);
D(P1--P4,dashed);

MC ("y",D((P4-(0.1,0))--(P4x-(0.1,0)),Arrow),0.5,W);
\end{asy}
\scalebox{.8}{
    \tikzsetnextfilename{pentplot}
    \begin{tikzpicture}
	\begin{axis}[
	ylabel style={rotate=-90},
		xlabel={$y$}, ylabel={$A$},
		xmin=-1., xmax=1.,
		ymin=1.60, ymax=1.62,
		ytick={1.60,1.61,1.62}
	    ]
	    \addplot[domain=-0.30540728933227923:0.30540728933227923] {1.6115786662088993+ 0.033061308882837155*x^2};
	    \addplot[domain=0.30540728933227923:1] {1.633850356689106 - 0.0797264910024892*x + 0.05533299936304392*x^2};
	    \addplot[domain=-1:-0.30540728933227923] {1.633850356689106 + 0.0797264910024892*x + 0.05533299936304392*x^2};
	\end{axis}
    \end{tikzpicture}
}
\caption{\label{fig:ninegon}
Half-length parallelograms in the regular 9-gon. That the minimum at $y=0$ is not the global minimum, and therefore that the densest
double-lattice packing is not given by the symmetric arrangement as for pentagons and heptagons, was possibly first noticed by
Graaf, Roij, and Dijkstra \cite{Graaf2011}.}
\end{center}
\end{figure}
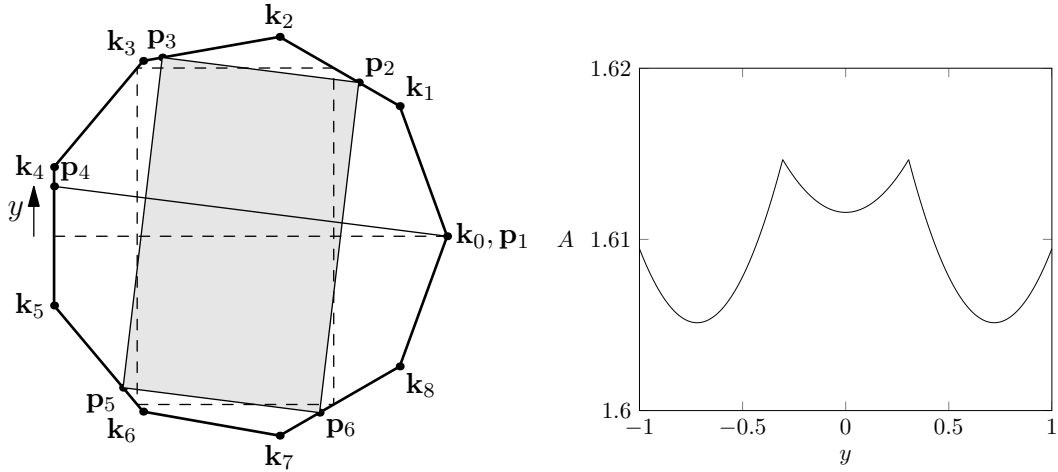

\Ref{fig:ninegon} illustrates some affine diameters and half-length parallelograms in the regular 9-gon.
We denote one endpoint of the affine diameter by $\mathbf{p}_1$ and the parallelogram vertices and other affine diameter endpoint as $\mathbf{p}_i$, $i=2,\ldots,6$
in counterclockwise order. The edges $\mathbf{p}_2\mathbf{p}_3$ and $\mathbf{p}_5\mathbf{p}_6$ will be the ones
parallel to the affine diameter.
The space of such labeled half-length parallelograms of a body $K$ is a circle, in the sense that we can
continuously parameterize $\mathbf{p}_i(t)$, $i=1,\ldots,6$, $t\in S^1$.  In fact, this parameterization can be piecewise linear \cite{mount1991densest}.
This parameterization specifies an affine diameter, a specification that does not change the half-length parallelogram or the generated double lattice. 
In the interior of each linear piece of the parameterization, either (1) one endpoint of the affine diameter is stationary at a vertex of $K$,
and the other points either move at a constant speed (possibly zero) along the interior of an edge or are stationary at a vertex
or (2) all parallelogram vertices are stationary while the affine diameter moves along two parallel edges. The piecewise linear
parameterization of augmented parallelograms can be converted to a piecewise linear parameterization of parallelograms by eliminating
the intervals of type 2. A half-length parallelogram is called \textit{pivotal}
if it sits at the boundary of two linear pieces, that is, there is a discontinuity in the direction of motion of at least one point $\mathbf{p}_i$.

In order to prove our main theorem in \ref{sec:genpoly}, we will need to know that the double lattice we wish to show
is strongly extreme does not have any of the vertices of the half-length parallelogram coincident with
vertices of the polygon $K$. This will generically follow from the fact that the double lattice packing is
an isolated minimum. In some exceptional cases, which must be treated separately, it does not.
We define the following exceptional types of parallelograms, illustrated in \ref{fig:except}:

\begin{figure}
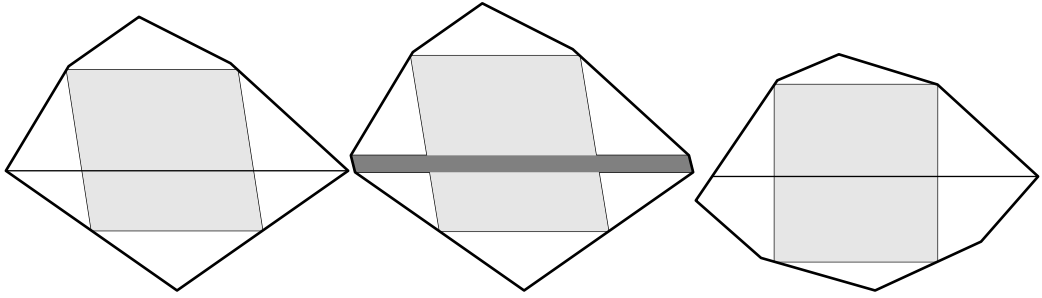

\begin{center}
\begin{asy}
import cseblack;
import olympiad;
usepackage("amssymb");
size(130);

path PP = D((96./271.,160./271.)--(-175./271.,160./271.)--(-1./2.,-7./20.)--(1./2.,-7./20.)--cycle);
fill(PP,lightgray);
path AD = D((-1,0)--(1,0));
path KK = D((1,0)--(5./16.,44./70.)--(-2./9.,9./10.)--(-19./30.,55./90.)--(-1.,0.)--(0,-7./10.)--cycle,linewidth(1));

\end{asy}
\begin{asy}
import cseblack;
import olympiad;
usepackage("amssymb");
size(130);

pair dx = (-1/40.,1/10.);
path PP = D(((96./271.,160./271.)+dx)--((-175./271.,160./271.)+dx)--(-1./2.,-7./20.)--(1./2.,-7./20.)--cycle);
path AD = D((-1,0)--(1,0)--((1,0)+dx)--((-1,0)+dx)--cycle);
fill(PP,lightgray);
fill(AD,gray);
path KK = D((1,0)--((1,0)+dx)--((5./16.,44./70.)+dx)--((-2./9.,9./10.)+dx)--((-19./30.,55./90.)+dx)--((-1,0)+dx)--(-1.,0.)--(0,-7./10.)--cycle,linewidth(1));

\end{asy}
\begin{asy}
import cseblack;
import olympiad;
usepackage("amssymb");
size(130);

path PP = D((13./34.,-89./170.)--(256./669.,1888./3345.)--(-413./669.,1888./3345.)--(-21./34.,-89./170.)--cycle);
fill(PP,lightgray);
path AD = D((-1.,0.)--(1.,0.));
path KK = D((1.,0.)--(256./669.,1888./3345.)--(-(2./9.), 3./4.)--(-(3./5.), 59./100.)--(-(11./10.), -(59./400.))--(-(7./10.), -(1./2.))--(0., -(7./10.))--(13./20., -(2./5.))--cycle,linewidth(1));

\end{asy}
\caption{\label{fig:except}
Three examples of the exceptional cases defined in \ref{def:except}. Left and middle: exceptional half-length parallelograms of type I.
Right: exceptional half-length parallelogram of type II.
}
\end{center}
\end{figure}

\begin{definition}\label{def:except}
    \begin{enumerate}
	\item If $\mathbf{p}_2$ and $\mathbf{p}_6$ are in the interiors of edges that meet at $\mathbf{p}_1$ and
	    $\mathbf{p}_3$ and $\mathbf{p}_5$ are in the interiors of edges that meet at $\mathbf{p}_4$, then
	    the parallelogram is exceptional of type I.
	\item If the affine diameter parallel to $\mathbf{p}_2\mathbf{p}_3$ is not unique,
	    then if $\mathbf{p}_2$ and $\mathbf{p}_3$ are in the interior of edges
	    that have as endpoints the endpoints of one such affine diameter
	    and $\mathbf{p}_5$ and $\mathbf{p}_6$ are in the interior of edges
	    that have as endpoints the endpoints of another such affine diameter, the
	    parallelogram will also be considered exceptional of type I.
	\item If $\mathbf{p}_3$ and $\mathbf{p}_4$ are in the interior of the same edge, and
	    $\mathbf{p}_2$ is at a vertex, then the parallelogram is exceptional
	    of type II. It can also be exceptional of type II if the same situation occurs
	    with $(\mathbf{p}_2,\mathbf{p}_1,\mathbf{p}_3)$,
	    $(\mathbf{p}_5,\mathbf{p}_4,\mathbf{p}_6)$, or 
	    $(\mathbf{p}_6,\mathbf{p}_1,\mathbf{p}_5)$ in place of 
	    $(\mathbf{p}_3,\mathbf{p}_4,\mathbf{p}_2)$.
    \end{enumerate}
\end{definition}

\begin{theorem}
    \label{thm:nopivot}
    When the half-length parallelogram $\mathbf{p}_2\mathbf{p}_3\mathbf{p}_5\mathbf{p}_6$ 
    is an isolated local minimum of area in the space of half-length parallelograms of a convex polygon $K$, then either (1) it is not pivotal, or (2)
    it is exceptional of type I.
\end{theorem}
\begin{proof}
    Assume that the given parallelogram is pivotal.
    If the parallel affine diameter is unique, then at least one endpoint
    is a vertex of $K$. If only one endpoint is a vertex let it be $\mathbf{p}_1$. Otherwise, let $\mathbf{p}_1$ be the
    vertex that makes the smaller angle between the adjoining counterclockwise edge and the affine diameter.
    If the affine diameter is not unique, all of the subsequent analysis will be equivalent if we remove
    from $K$ the parallelogram traced out by the affine diameters and identify the two extreme affine diameters.
    In the modified version of $K$, the affine diameter is unique.
    We may pick our orientation, scale, and origin without loss of generality
    so that the affine diameter is horizontal, of length 2, and bisected by the origin.

\begin{figure}
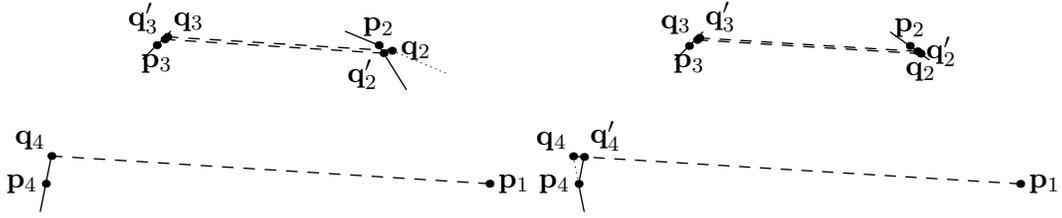

\begin{center}
\begin{asy}
import cseblack;
import olympiad;
usepackage("amssymb");
size(200);

path E4 = D((-.82,-0.1)--(-0.78,0.1));
pair P1 = D("\mathbf{p}_1",(0.8,0),E);
pair P4 = D("\mathbf{p}_4",(-0.8,0),W);
path E2 = D((0.28,0.55)--(0.4,0.5));
path E2N = D((0.64,0.4)--(0.4,0.5),dotted);
path D2 = D((0.5,0.34)--(0.4,0.5));
pair P2 = D("\mathbf{p}_2",(0.4,0.5),N);
path E3 = D((-0.45,0.45)--(-0.35,0.55));
pair P3 = D("\mathbf{p}_3",(-0.4,0.5),S);
path AD = D((-0.78,0.1)--(0.8,0),dashed);
pair Q4 = D("\mathbf{q}_4",(-0.78,0.1),NW);
pair Q2 = D("\mathbf{q}_2",(0.448, 0.48),E);
pair Q3 = D("\mathbf{q}_3",(-0.362, 0.53),NE);
path QQ = D(Q2--Q3,dashed);
pair Q2p = D("\mathbf{q}'_2",(0.418116, 0.471014),SW);
pair Q3p = D("\mathbf{q}'_3",(-0.371884, 0.521014),NW);
path QQp = D(Q2p--Q3p,dashed);

\end{asy}
\begin{asy}
import cseblack;
import olympiad;
usepackage("amssymb");
size(200);

path E4 = D((-.78,-0.1)--(-0.8,0.));
path E4N = D((-.8,-0.)--(-0.82,0.1),dotted);
path D4 = D((-.78, 0.1)--(-0.8,0.));
pair P1 = D("\mathbf{p}_1",(0.8,0),E);
pair P4 = D("\mathbf{p}_4",(-0.8,0),W);
path E2 = D((0.33,0.55)--(0.47,0.45));
pair P2 = D("\mathbf{p}_2",(0.4,0.5),N);
path E3 = D((-0.45,0.45)--(-0.35,0.55));
pair P3 = D("\mathbf{p}_3",(-0.4,0.5),S);
path AD = D((-0.82,0.1)--(0.8,0),dashed);
pair Q4 = D("\mathbf{q}_4",(-0.82,0.1),NW);
pair Q4p = D("\mathbf{q}'_4",(-0.780488, 0.097561),NE);
pair Q2 = D("\mathbf{q}_2",(0.4392, 0.472),S);
pair Q3 = D("\mathbf{q}_3",(-0.3708, 0.522),NW);
path QQ = D(Q2--Q3,dashed);
pair Q2p = D("\mathbf{q}'_2",(0.427385, 0.480439),E);
pair Q3p = D("\mathbf{q}'_3",(-0.362859, 0.52922),NE);
path QQp = D(Q2p--Q3p,dashed);

\end{asy}
\caption{\label{fig:backevol}
Extending the forward evolution of the half-length parallelogram to negative times at a pivotal parallelogram. Left: when one
of the parallelogram vertices becomes external, we construct the parallel chord of the same length. Right: when an endpoint
of the affine diameter becomes external and the intersection with the polygon is still an affine diameter, we construct
the two parallel chords of half the length of that affine diameter (bottom chord omitted).
}
\end{center}
\end{figure}

    We can evolve the parameterization forward (counterclockwise) or backward (clockwise) from the given parallelogram.
    For each $i=1,\ldots,6$, let $\mathbf{v}_i=(\cos\phi_i,\sin\phi_i)$ be a unit vector pointing, in the counterclockwise direction,
    along the edge of $K$ that $\mathbf{p}_i$ belongs to. If $\mathbf{p}_i$ is at the intersection of two edges,
    we use the counterclockwise adjoining edge.
    Let $\mathbf{u}_i=(\cos\chi_i,\sin\chi_i)$ be defined exactly the same way, except
    we use the clockwise adjoining edge in case of intersection.
    The forward evolution is by velocities $c_i$, $i=1,\ldots,6$ satisfying
    \begin{equation}\label{eq:evol}
	c_1 \mathbf{v}_1 - c_4 \mathbf{v}_4 = 2 c_2\mathbf{v}_2-2 c_3\mathbf{v}_3 = 2 c_6\mathbf{v}_6-2 c_5\mathbf{v}_5\text.
    \end{equation}
    Together with the fact that $c_1=0$, which follows from our choice of which affine diameter endpoint to label as $\mathbf{p}_1$,
    this equation determines $c_i$, $i=1,\ldots,6$, up to multiplication by a common factor. The points evolves as $\mathbf{q}_i(t) = \mathbf{p}_i+t c_i\mathbf{v}_i$.

    The area of the parallelogram is a quadratic function of $t$.
    If its slope for forward evolution is negative, then the parallelogram is not a local minimum, and
    we are done. Therefore, let it be nonnegative. We now extend the forward evolution to negative times.
    We obtain parallelograms, not necessarily inscribed, of area $A(t) \le A(0) +O(t^2)$, together with parallel segments not
    necessarily affine diameters, but double the length of the parallel sides.
    Since the initial parallelogram is pivotal, at least one point $\mathbf{q}_i(t)$, $i=2,\ldots,5$ will go into the exterior of $K$,
    and the distance between it and $K$ will grow linearly with $t$.
    We will argue that from these parallelograms, we can construct half-length parallelograms of area $A'(t) \le A(t) + a t$,
    with $a>0$. When we have shown those exist, it follows that the given parallelogram is not a local minimum.

    We first treat the case where exactly one point becomes exterior. Suppose the point that becomes exterior is $\mathbf{q}_2(t)$.
    We find the points $\mathbf{q}_2'(t) = \mathbf{p}_2 + t c_2'\mathbf{u}_2$ and $\mathbf{q}_3'(t) = \mathbf{p}_3 + t c_3'\mathbf{v}_3$
    that form a segment that is parallel to and of the same length as $\mathbf{q}_2(t)\mathbf{q}_3(t)$. A trigonometric calculation
    gives $c_2'= c_2 \sin(\phi_3-\phi_2)/\sin(\phi_3-\chi_2)$ and $c_3'=c_3 - c_2\sin(\phi_2-\chi_2)/\sin(\phi_3-\chi_2)$.
    At least up to some finite negative time $-T\le t<0$, $\mathbf{q}_2'(t)$ is on the clockwise edge adjoining $\mathbf{p}_2$ ($c_2'>0$),
    and $\mathbf{q}_3(t)$ is either on the interior of the same edge as $\mathbf{p}_3$, if the latter is in the interior of an edge,
    or else, if $\mathbf{p}_3$ is at a vertex and stationary for forward evolution, then $c_3=0$ and $c_3'<0$, so
    $\mathbf{q}_3'(t)$ is on the counterclockwise adjoining edge to $\mathbf{p}_3$. In any case,
    the parallelogram $\mathbf{q}_2'(t)\mathbf{q}_3'(t)\mathbf{q}_5(t)\mathbf{q}_6(t)$ is inscribed (therefore
    a half-length parallelogram), and its
    area is $A'(t) = A(t) - c_2 t \sin \phi_3 \sin(\phi_2-\chi_2)/\sin(\phi_3-\chi_2) + O(t^2)$.
    If $\phi_3>\pi$, the linear term is negative for negative times, and we have our desired parallelogram.
    Otherwise, $\phi_3=\pi$ and the edge $\mathbf{p}_2\mathbf{p}_3$ can be moved left while staying in $K$, so the given
    parallelogram is not an isolated minimum.
    The same argument works, \textit{mutatis mutandis}, when any one of the other three parallelogram vertices becomes exterior.

    Now suppose that $\mathbf{q}_4(t)$ becomes exterior for $t<0$. We distinguish two cases, (i) if $\chi_4-\chi_1\ge \pi$,
    then $\mathbf{p}_1\mathbf{q}'_4$, the intersection of $\mathbf{p}_1\mathbf{q}_4(t)$ with $K$, is an affine diameter, and (ii) otherwise,
    there is a parallel affine diameter $\mathbf{q}'_1(t)\mathbf{p}_4$, with $\mathbf{q}'_1(t)$ on the clockwise adjoining edge
    to $\mathbf{p}_1$. In case (i), we have $\mathbf{q}'_4(t) = \mathbf{p}_4 + l_4(t) \mathbf{u}_4$,
    where $l_4(t) = c_4 t \sin\phi_4/(\sin\chi_4 + \tfrac12 c_4 t \sin(\phi_4-\chi_4))$.
    We find points $\mathbf{q}_i'(t) = \mathbf{q}_i(t) + l_i(t) \mathbf{v}_i$, $i=2,5$ and $\mathbf{q}_i'(t) = \mathbf{q}_i(t) + l_i(t) \mathbf{u}_i$, $i=3,6$,
    such that $2(\mathbf{q}_2'(t) - \mathbf{q}_3'(t)) = 2(\mathbf{q}_6'(t) - \mathbf{q}_5'(t)) = \mathbf{p}_1 - \mathbf{q}_4'(t)$.
    For small enough negative times, this construction will yield a half-length parallelogram, and its area is
    \begin{equation}
	A'(t) = A(t) + \tfrac12c_4 \frac{\sin(\phi_4-\chi_4)}{\sin\chi_4} \left[ h - \frac{1}{\cot\chi_3-\cot\phi_2} - \frac{1}{\cot\chi_6-\cot\phi_5}\right] + O(t^2)\text,
    \end{equation}
    Where $h$ is the height of the parallelogram.
    Since we chose $2$ for the length of the affine diameter $\mathbf{p}_1\mathbf{p}_4$, we have $h\le\frac{1}{\cot\chi_3-\cot\phi_2} + \frac{1}{\cot\chi_2-\cot\phi_5}$,
    with equality only if the parallelogram is exceptional of type I. In case (ii), we have
    $\mathbf{q}'_1(t) = \mathbf{p}_1 + l_1(t)\mathbf{u}_1$, 
    where $l_1(t) = c_4 t \sin\phi_4/(\sin(\chi_1+\pi) + \tfrac12 c_4 t \sin(\phi_4-\chi_1-\pi))$.
    Again we find the points $\mathbf{q}_i'(t)$, $i=2,3,5,6$, that make up a half-length parallelogram
    with the new affine diameter, and the area of this parallelogram is
    \begin{equation}
	A'(t) = A(t) + \tfrac12c_4 \frac{\sin(\phi_4-\chi_1-\pi)}{\sin(\chi_1-\pi)} \left[ h - \frac{1}{\cot\chi_3-\cot\phi_2} - \frac{1}{\cot\chi_6-\cot\phi_5}\right] + O(t^2)\text.
    \end{equation}
    In this case too, the given parallelogram is either not a local minimum or is exceptional of type I.

    To handle the case of multiple points becoming exterior upon extension of the forward evolution to negative times,
    we can simply compose the operations we performed in each of the cases of a single exterior point.
    In each step of the composed operation, we pretend that the exterior points we have not yet handled
    and we are not handling right now are actually on the boundary of the polygon. Since in each step
    of the operation we obtain a linear-order reduction to the area of the parallelogram, the composition
    also provides a linear-order reduction. Therefore, the given parallelogram cannot be an isolated
    local minimum.
\end{proof}

\begin{theorem}
    \label{thm:novtx}
    If the half-length parallelogram $\mathbf{p}_2\mathbf{p}_3\mathbf{p}_5\mathbf{p}_6$
    is not pivotal and is an isolated local minimum of area in the space of half-length parallelograms of a convex polygon $K$ then either (1)
    all its vertices and at least one affine diameter endpoint are in the interior of polygon edges, or (2) it is exceptional of type II.
\end{theorem}
\begin{proof}
    Since the parallelogram is not pivotal, forward and backward evolution use the same linear velocities. The solution (up to common factor)
    to \ref{eq:evol} is
    \begin{equation}
	\begin{aligned}
	    c_1 &= 0\\
	    c_2 &= \sin(\phi_4-\phi_3)\sin(\phi_6-\phi_5)\\
	    c_3 &= \sin(\phi_4-\phi_2)\sin(\phi_6-\phi_5)\\
	    c_4 &= 2\sin(\phi_3-\phi_2)\sin(\phi_6-\phi_5)\\
	    c_5 &= \sin(\phi_3-\phi_2)\sin(\phi_6-\phi_4)\\
	    c_6 &= \sin(\phi_3-\phi_2)\sin(\phi_5-\phi_4)\text.
	\end{aligned}
    \end{equation}
    When the affine diameter is not unique, we perform the same modification as in the previous proof.
    The only way one of the labeled points except $\mathbf{p}_1$ can be at a vertex, is if its velocity $c_i$ vanishes.
    If $\phi_3=\phi_2$ or $\phi_5=\phi_6$, then one of the parallelogram edges is strictly contained in a polygon edge and the parallelogram
    is not a local minimum. Therefore, the only way a labeled point other than $\mathbf{p}_1$ can be stationary is if
    $\phi_4=\phi_2+\pi,\phi_3,\phi_5$, or $\phi_6-\pi$. If $\phi_4=\phi_2+\pi$, then necessarily also $\phi_1=\phi_2$, the affine
    diameter is not unique, reaching a contradiction. Similarly, we do not have $\phi_4=\phi_6-\pi$.
    If $\phi_4=\phi_3$ and $\mathbf{p}_2$ is stationary at a vertex or if $\phi_4=\phi_5$ and $\mathbf{p}_5$ is stationary at a vertex
    then the parallelogram is exceptional of type II.
\end{proof}

\subsection{Honeycomb construction}

We now describe a honeycomb associated with any double lattice packing. Let $K$ be a
convex polygon and let $\mathbf p_2\mathbf p_3\mathbf p _5\mathbf p_6$ be a half-length parallelogram,
such that $\mathbf p_2\mathbf p_3$ and $\mathbf p_6\mathbf p_5$ are half the length of and parallel to
the affine diameter $\mathbf p_1\mathbf p_4$. The double lattice generated by reflections about the
vertices of the parallelogram is $\Xi$ and the subgroup of translations is the lattice $\Lambda$.
Let $\xi_0 = \mathrm{Id}$, $\xi_1 = \mathrm{Tran}_{\mathbf p_1-\mathbf p_4}$, $\xi_2 = \mathrm{Ref}_{\mathbf{p}_2}$,
and $\xi_6 = \mathrm{Ref}_{\mathbf{p}_6}$, where $\mathrm{Ref}_\mathbf{r}$ is a reflection about $\mathbf{r}$
and $\mathrm{Tran}_\mathbf{r}$ is a translation by $\mathbf{r}$.
Let $s_2=\{\xi_0(0),\xi_1(0),\xi_2(0)\}$ and $s_6=\{\xi_0(0),\xi_6(0), \xi_1(0)\}$,
then $\mathcal{T}_2 = \{ \xi(s_2), \xi(s_6): \xi\in\Xi\}$
are the full-dimensional simplices of a triangulation $\mathcal T$ of $\mathbb{R}^2$, and
if $p(\xi(s_2))=p(\xi(s_6))=\xi$, then $(\mathcal T,p)$ is a honeycomb of the double lattice $\Xi$.
The optimization problem
of minimizing $\mathrm{vol}\,\mathcal P'_\phi$ over $(\cdot)':\Xi_\phi\to E(n)$, is equivalent for every $\phi\in\Xi$.
Therefore, to show that \ref{thm:noaux} applies, it suffices to show that the restriction of the identity
is optimal for the problem associated with the cell $\mathcal P_{\xi_0}$.

\begin{figure}
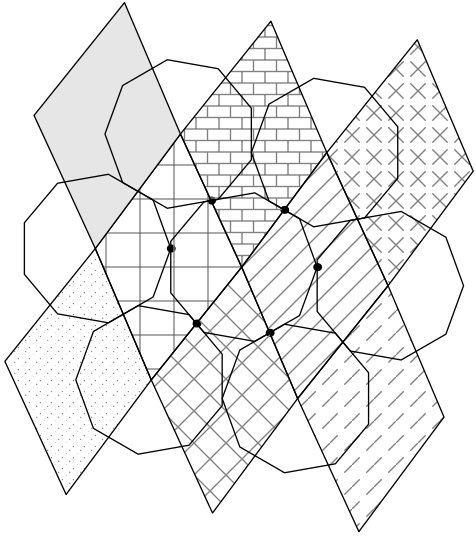

\begin{center}
\begin{asy}
import cseblack;
import olympiad;
import patterns;
usepackage("amssymb");
size(200);
pair K0=(1,0);
pair K1=rotate(360./9.)*K0;
pair K2=rotate(360./9.)*K1;
pair K3=rotate(360./9.)*K2;
pair K4=rotate(360./9.)*K3;
pair K5=rotate(360./9.)*K4;
pair K6=rotate(360./9.)*K5;
pair K7=rotate(360./9.)*K6;
pair K8=rotate(360./9.)*K7;
path KK = K0--K1--K2--K3--K4--K5--K6--K7--K8--cycle;
pair P1=D((1., 0.));
pair P2=D((0.3706844903100875, -0.87104878486755));
pair P3=D((-0.5991618200828668, -0.7478489484525013));
pair P4=D((-0.9396926207859084, 0.24639967283009717));
pair P5=D((-0.4058322193092865, 0.882629724233649));
pair P6=D((0.5640140910836678, 0.7594298878186004));
path PP = (0,0)--(2*P6)--(P1-P4)--(2*P2)--cycle;
add("p1",hatch(2mm,gray));
add("p2",tile(4.5mm,gray));
add("p3",crosshatch(3mm,gray));
add("p4",brick(3mm,gray));
add("p5",hatch(2mm,gray+dashed));
add("p6",crosshatch(2mm,gray+dashed));
add("p7",hatch(1mm,gray+dotted));
filldraw(PP,pattern("p1"));
filldraw(shift(P4-P1)*PP,pattern("p2"));
filldraw(shift(2*P2)*scale(-1)*PP,pattern("p3"));
filldraw(shift(2*P6)*scale(-1)*PP,pattern("p4"));
filldraw(shift(P1-P4)*shift(2*P2)*scale(-1)*PP,pattern("p5"));
filldraw(shift(P1-P4)*shift(2*P6)*scale(-1)*PP,pattern("p6"));
filldraw(shift(P4-P1)*shift(2*P2)*scale(-1)*PP,pattern("p7"));
filldraw(shift(P4-P1)*shift(2*P6)*scale(-1)*PP,lightgray);
D(KK);
D(shift(P1-P4)*KK);
D(shift(P4-P1)*KK);
D(shift(2*P2)*(scale(-1)*KK));
D(shift(2*P3)*(scale(-1)*KK));
D(shift(2*P5)*(scale(-1)*KK));
D(shift(2*P6)*(scale(-1)*KK));
\end{asy}
\end{center}
\caption{Honeycomb construction for the regular 9-gon. Each honeycomb cell (hatched parallelogram) is composed of
two triangles. The area minimization problem for each cell involves only the four 9-gons centered
at the cell vertices.}
\end{figure}

For every convex body $K$ and double lattice $\Xi$ we now have a concrete optimization problem
to solve: we wish to minimize the area of the quadrilateral $\xi'_0(0)\xi'_6(0)\xi'_1(0)\xi'_2(0)$
subject to the constraints that $\xi'_i(K)$ and $\xi'_j(K)$ do not overlap. Since the objective
and the constraints are invariant under common isometry, we may fix $\xi'_i=\xi_i$ for
one $i$. We parametrize
$\xi_i' = \mathrm{Tran}_{\mathbf{r}_i}\xi_i \mathrm{Rot}_{\theta_i}$, where $\mathrm{Rot}_\theta$ is a rotation
by $\theta$ about the origin. Since we are only
interested in certifying that the initial configuration is a local minimum, we can replace
the constraints with ones that are equivalent in a neighborhood.

\begin{lemma}\label{lemma:edgedg}
    Let $K$ and $K'$ be two polygons that intersect at a segment, which is not
    identical with a full edge of $K$ or of $K'$. The endpoints of the segments
    are $\mathbf{x}$ a vertex of $K$ and $\mathbf{y}$ a vertex of $K'$. Let $\mathbf{y}\mathbf{y}'$
    and $\mathbf{x}\mathbf{x}'$ be the edges of $K$ and $K'$ containing the intersection.
    Let $\mathbf{x}'\mathbf{y}\mathbf{x}\mathbf{y}'$ be oriented counterclockwise
    from the point of view of the interior of $K$ (otherwise switch $K$ and $K'$).
    There is some $\epsilon>0$ such that whenever $N(\xi),N(\xi')<\epsilon$,
    then $\xi(K)$ and $\xi'(K')$ have disjoint interiors if and only if
    $\alpha(\xi(\mathbf{x})\xi(\mathbf{x}')\xi'(\mathbf{y}))\ge0$
    and $\alpha(\xi'(\mathbf{y}')\xi'(\mathbf{y})\xi(\mathbf{x}))\ge0$,
    where $\alpha$ is the signed area of the oriented triangle.
\end{lemma}

\begin{lemma}\label{lemma:edgver}
    Let $K$ and $K'$ be two polygons that intersect at a point and not at a segment.
    The intersection point $\mathbf y$ is a vertex of one polygon, which we let be $K'$,
    and sits in the relative interior of the segment $\mathbf{x}'\mathbf{x}$, which in turn
    is contained in an edge of $K$. We assume the segment $\mathbf{x}'\mathbf{x}$ is
    oriented counterclockwise from the point of view of the interior of $K$.
    There is some $\epsilon>0$ such that whenever $N(\xi),N(\xi')<\epsilon$,
    then $\xi(K)$ and $\xi'(K')$ have disjoint interiors if and only if
    $\alpha(\xi(\mathbf{x})\xi(\mathbf{x}')\xi'(\mathbf{y}))\ge0$.
\end{lemma}

Note that two cases are not treated: the case of an intersection at a point that is a vertex of both polygons
and the case of an intersection at a full edge of one or both polygons.
In the optimal double-lattice packings of the first two bodies we treat, the regular pentagon
and the regular heptagon, there are no such intersections. For the application
to more general convex polygons, we have already shown that the first intersection
case only arises in exceptional cases. The constraint arising in the second case
can be relaxed to a constraint of the form of \ref{lemma:edgedg} without affecting
the solution.

We will show that optimization problems that arise fall into a convenient
form, where linear stability holds along all but one direction. Along
the direction of vanishing linear stability, the theory of
Kuperberg and Kuperberg will be shown to guarantee stability.

Consider the nonlinear optimization problem
\begin{equation}\label{eq:opt-nonlin}
    \begin{aligned}
\text{minimize } & f(x)\text,\\
\text{over } & x\in \mathbb{R}^n\text, \\
\text{subject to } & g_r(x)\ge 0, r\in I\text,\\
& \|x\|\le\epsilon\text.
    \end{aligned}
\end{equation}
We will prove that the following conditions are sufficient for the origin to be the unique solution of \ref{eq:opt-nonlin}
for some $\epsilon>0$.
Let $e_1$ denote the standard unit vector $(1, 0,\dots,0)\in\mathbb{R}^n$, $E=\{te_1\}$
its span, and $H$ the orthogonal complement, so that $\mathbb{R}^n = E\oplus H$.

\begin{conditions}\label{cnd:assumpt}
\quad
\begin{enumerate}
\item \label{itm:finite} $I$ is a finite set.
\item \label{itm:contdiff} $f(x)$ and $g_r(x)$, $r\in I$, are continuously differentiable. Denote their
    derivatives $F(t) = \nabla f (te_1)$ and $G_r(t) = \nabla g_r(te_1)$.
\item \label{itm:origin} $f(0) = g_r(0) = 0$ for all $r$ in $I$.
\item \label{itm:lp} The linear program
    \begin{equation}\label{eq:lp}
	\text{minimize }_{x\in \mathbb{R}^n}F(0)\cdot x \textrm{ subject to }G_r(0)\cdot x \ge 0, r\in I
    \end{equation}
has $E$ as the set of solutions.
\item \label{itm:zero} There is an $\epsilon > 0$ so the functions $g_r(te_1) = 0$ for all $-\epsilon<t<\epsilon$, $r\in I$.
\item \label{itm:min} $f(te_1)$ has an isolated local minimum at $t=0$.
\end{enumerate}
\end{conditions}

\begin{theorem}\label{thm:slice}
    Given Conditions \ref{cnd:assumpt}, there exists $\epsilon>0$ such that the origin is the unique solution of \ref{eq:opt-nonlin}.
\end{theorem}
\begin{proof}
Consider the sliced problem with $t$ fixed,
\begin{equation}\label{eq:opt-slice}
    \begin{aligned}
	\text{minimize } & f(x)\text,\\
	\text{over } & x\in H+t e_1\text,\\
\text{subject to } & g_r(x)\ge 0, r\in I\\
& \|x\|\le\epsilon\text.
    \end{aligned}
\end{equation}
For $t=0$, the local minimum at $x=0$ is guaranteed by the LP \ref{eq:lp} \cite{bazaraa2006nonlinear}. The LP implies, by duality, that
$F(0)$ lies in the interior (relative to $H$) of the cone finitely generated by $G_r(0)$, $r\in I$. By continuity,
we have that $F(t)$ also lies in the interior of the cone generated by $G_r(t)$, $r\in I$, for all $|t|<\epsilon$, for
some $\epsilon>0$. Therefore, in each slice, the minimum of \ref{eq:opt-slice} occurs at $x=te_1$.
From the final condition of \ref{cnd:assumpt}, we conclude that the origin is the unique solution of \ref{eq:opt-nonlin}
for some $\epsilon>0$.
\end{proof}

\section{Calculations}

\subsection{Pentagons}

Let us fix a regular pentagon $K=\mathrm{conv} \{\mathbf{k}_i:i=0,\ldots 4\}$, where $\mathbf{k}_i = \mathrm{Rot}_{2\pi i/5} (1,0)$.
In this subsection, we do all the calculations in the extension field $\mathbb{Q}(u,v)$, where $u=\cos \pi/5$ and $v=\sin \pi/5$.

\begin{figure}
    \begin{center}
\begin{asy}
import cseblack;
import olympiad;
usepackage("amssymb");
size(200);
pair K0=D("\mathbf{k}_0,\mathbf{p}_1",(1,0),E);
pair K1=D("\mathbf{k}_1",rotate(360./5.)*K0,NE);
pair K2=D("\mathbf{k}_2",rotate(360./5.)*K1,NW);
pair K3=D("\mathbf{k}_3",rotate(360./5.)*K2,SW);
pair K4=D("\mathbf{k}_4",rotate(360./5.)*K3,SE);
path KK = D(K0--K1--K2--K3--K4--cycle,linewidth(1));

real u = cos(pi/5.);
pair P1=K0;
pair P2=D("\mathbf{p}_2",(K0+3.*K1)/4.,NE);
pair P3=D("\mathbf{p}_3",((3.-2.*u)*K1+(1.+2.*u)*K2)/4.,NW);
pair P4=D("\mathbf{p}_4",(2.*K2+2.*K3)/4.,SW);
pair P5=D("\mathbf{p}_5",((3.-2.*u)*K4+(1.+2.*u)*K3)/4.,SW);
pair P6=D("\mathbf{p}_6",(K0+3.*K4)/4.,SE);
path PP = P2--P3--P5--P6--cycle;
fill(PP,lightgray);
D(PP);
D(P1--P4);

pair V2 = 1.*(K1-P2);
pair V3 = (1./sqrt(5.))*(K2-P3);
pair V4 = 1.*(K3-P4);
pair V5 = (sqrt(5.)-2.)*(K4-P5);
pair V6 = (1./3.)*(K0-P6);
real t = -0.6;
path PP = D((P2+t*V2)--(P3+t*V3)--(P5+t*V5)--(P6+t*V6)--cycle,dashed);
D( P1--(P4+t*V4),dashed);

MC ("y",D((P4-(0.1,0))--(P4+t*V4-(0.1,0)),Arrow),0.5,W);
\end{asy}
\scalebox{.8}{
    \tikzsetnextfilename{pentplot}
    \begin{tikzpicture}
	\begin{axis}[
	ylabel style={rotate=-90},
		xlabel={$y$}, ylabel={$A$},
		xmin=-1., xmax=1.,
		ymin=1.25, ymax=1.4,
	    ]
	    \addplot[] {1.2903580504417251 + 0.10153740507278321*x^2};
	\end{axis}
    \end{tikzpicture}
}
\caption{\label{fig:pent}
Half-length parallelograms in the regular pentagon.}
\end{center}
\end{figure}
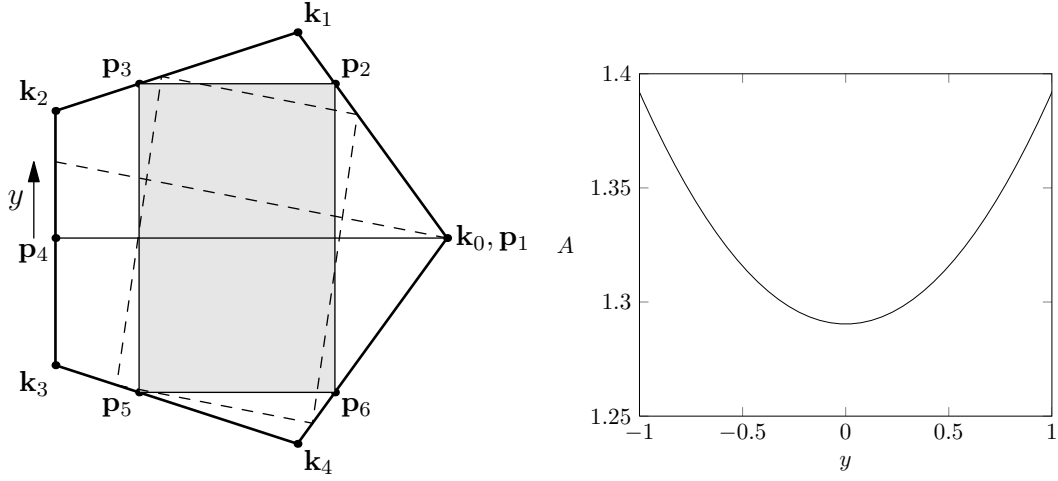

One minimum-area half-length parallelogram corresponds to the affine diameter $\mathbf{p}_1\mathbf{p}_4$, where
$\mathbf{p}_1 = \mathbf{k}_0$ and $\mathbf{p}_4 = \tfrac12 (\mathbf{k}_2+\mathbf{k}_3)$.
The vertices of the parallelogram are given by
$\mathbf{p}_2=\tfrac14\mathbf k_0+\tfrac34\mathbf k_1$,
$\mathbf{p}_3=\tfrac{3-2u}4\mathbf k_1+\tfrac{1+2u}4\mathbf k_2$,
$\mathbf{p}_5=\tfrac{1+2u}4\mathbf k_3+\tfrac{3-2u}4\mathbf k_4$, and
$\mathbf{p}_6=\tfrac34\mathbf k_4+\tfrac14\mathbf k_0$.
This half-length parallelogram is illustrated in \ref{fig:pent}.

The four pentagons that surround our primitive honeycomb cell are
$\xi_i(K)$, $i=0,1,2,6$, where $\xi_0=\mathrm{Id}$,
$\xi_1=\mathrm{Tran}_{\mathbf{p}_1-\mathbf{p}_4}$,
$\xi_2=\mathrm{Ref}_{\mathbf{p}_2}$,
and $\xi_6=\mathrm{Ref}_{\mathbf{p}_6}$.
We are interested in showing that the assignment $\xi_i'=\xi_i$, $i=0,1,2,6$, locally
minimizes the area of the quadrilateral $\xi_0(0)\xi_6(0)\xi_1(0)\xi_2(0)$, subject
to the nonoverlap constraints. As explained in the previous section, we may
fix $\xi'_1=\xi_1$ and replace the nonoverlap constraints by signed area constraints. We
obtain the following optimization problem:
\begin{equation}\label{eq:opt-pent}
    \begin{aligned}
	\text{minimize } & f(z)=\alpha(\xi'_0(0),\xi'_1(0),\xi'_2(0))+\alpha(\xi'_0(0),\xi'_6(0),\xi'_1(0) )\text, \\
	\text{subj.\ to } & g_1(z)=\alpha(\xi'_0(\mathbf{k}_1),\xi'_0(\mathbf{k}_0),\xi'_2(\mathbf{k}_1))\ge0\\
	                & g_2(z)=\alpha(\xi'_2(\mathbf{k}_1),\xi'_2(\mathbf{k}_0),\xi'_0(\mathbf{k}_1))\ge0\\
	                & g_3(z)=\alpha(\xi'_0(\mathbf{k}_0),\xi'_0(\mathbf{k}_4),\xi'_6(\mathbf{k}_4))\ge0\\
	                & g_4(z)=\alpha(\xi'_6(\mathbf{k}_0),\xi'_6(\mathbf{k}_4),\xi'_0(\mathbf{k}_4))\ge0\\
	                & g_5(z)=\alpha(\xi'_1(\mathbf{k}_2),\xi'_1(\mathbf{k}_1),\xi'_2(\mathbf{k}_2))\ge0\\
	                & g_6(z)=\alpha(\xi'_2(\mathbf{k}_2),\xi'_2(\mathbf{k}_1),\xi'_1(\mathbf{k}_2))\ge0\\
	                & g_7(z)=\alpha(\xi'_1(\mathbf{k}_4),\xi'_1(\mathbf{k}_3),\xi'_6(\mathbf{k}_3))\ge0\\
	                & g_8(z)=\alpha(\xi'_6(\mathbf{k}_4),\xi'_6(\mathbf{k}_3),\xi'_1(\mathbf{k}_3))\ge0\\
	                & g_9(z)=\alpha(\xi'_1(\mathbf{k}_3),\xi'_1(\mathbf{k}_2),\xi'_0(\mathbf{k}_0))\ge0\text,
    \end{aligned}
\end{equation}
where $\xi'_i=\mathrm{Tran}_{(x_i,y_i)}\xi_i\mathrm{Rot}_\theta$ for $i=0,2,6$, and $\xi'_1=\xi_1$. We adopt
a condensed notation for the free variables $z=(x_0,y_0,\theta_0,x_2,y_2,\theta_2,x_6,y_6,\theta_6)$.

We consider the linearization of \ref{eq:opt-pent} around the point $z=0\in\mathbb{R}^9$. This gives a
problem of the form
\begin{equation}\label{eq:lin-pent}
    \text{minimize } c\cdot z \text{ subject to } G z\ge 0\text,
\end{equation}
where $c\in \mathbb R^9$, $G\in\mathbb{R}^{9\times9}$ and we use the line programming notation $\ge0$
to denote a vector lying in the closed positive orthant. The exact numeric values of $G$ and $c$ are given in
\ref{tab:pentGc}.
We can show by direct calculation a vector $\eta>0$ lying in the open positive orthant exists such that
$c=\eta^T G$. Such a vector is given in \ref{tab:pentGc}.
By the fundamental theorem of linear algebra, this observation implies that $G z\ge 0$ and $c\cdot z\le0$
if and only if $G z = 0$ and $c\cdot z=0$, and so the program
\ref{eq:lin-pent} is minimized exactly by the null space of $G$ and is suboptimal elsewhere in the
cone $Gz\ge0$. We calculate the rank of $G$ to be $8$, and so the null space is one-dimensional,
and it is generated by the vector $z_0$ given in \ref{tab:pentGc}.
The null space corresponds precisely to the rearrangement given by evolving
the half-length parallelogram according to the locally linear parameterization
discussed in \ref{sec:doublatt}. Importantly, this motion involves no rotations. We can verify directly
that $f(tz_0)$ is a quadratic function of $t$ minimized at $t=0$, and that
$g_r(tz_0)=0$ identically for $r=1,\ldots,9$. Indeed, perturbing the half-length
parallelogram away from the minimum-area one increases the area of the resulting cell
and maintains all the contacts.

\begin{table}
    \begin{equation*}
	\begin{aligned}
	    G &= \left(\begin{array}{ccccccccc}-2 u v & -\tfrac32 + u & 0 & 2 u v & \tfrac32 - u & \tfrac32 - u & 0 & 0 & 
      0 \\ -2 u v & -\tfrac32 + u & \tfrac32 - u & 2 u v & \tfrac32 - u & 0 & 0 & 0 & 
      0 \\ -2 u v & \tfrac32 - u & 0 & 0 & 0 & 0 & 
      2 u v & -\tfrac32 + u & -\tfrac32 + u \\ -2 u v & \tfrac32 - u & -\tfrac32 + u & 0 & 0 & 
      0 & 2 u v & -\tfrac32 + u & 0 \\ 0 & 0 & 0 & v - 2 u v & -\tfrac12 + 2 u & 
      \tfrac32 - u & 0 & 0 & 0 \\ 0 & 0 & 0 & v - 2 u v & -\tfrac12 + 2 u & -\tfrac72 + 4 u & 
      0 & 0 & 0 \\ 0 & 0 & 0 & 0 & 0 & 0 & v - 2 u v & \tfrac12 - 2 u & -\tfrac32 + u \\ 0 & 
      0 & 0 & 0 & 0 & 0 & v - 2 u v & \tfrac12 - 2 u & \tfrac72 - 4 u \\ -2 v & 0 & 0 & 0 & 0 & 
      0 & 0 & 0 & 0\end{array}\right)\\
      \eta^T &= \left(\begin{array}{ccccccccc}\tfrac14, & \tfrac14, & \tfrac14, & \tfrac14, & \tfrac1{10} + \tfrac{u}{10}, & \tfrac25 + \tfrac9{10}u, &
      \tfrac1{10} + \tfrac{u}{10}, \tfrac25 + \tfrac9{10}u, 2 u \end{array}\right)\\
	    c^T &= \left(\begin{array}{ccccccccc}-6 u v, & 0, & 0, & 0, & 1 + u, & 0, & 0, & -1 - u, & 0\end{array}\right)\\
	    z_0^T &= \left(\begin{array}{ccccccccc}0, & 2 + 4 u, & 0, & 2 v + 4 u v, & 1, & 0, & -2 v - 4 u v, & 1, & 0\end{array}\right)
	\end{aligned}
  \end{equation*}
  \caption{\label{tab:pentGc}Constraint Jacobian and objective gradient for the pentagon
  packing problem.}
\end{table}

Therefore, \ref{eq:opt-pent}, after a change in coordinates, satisfies all the conditions of \ref{thm:slice}, and following
directly from \ref{thm:noaux} we have:

\begin{theorem}
    The optimal double-lattice packing of regular pentagons, illustrated in \ref{fig:pent}, is strongly extreme.
\end{theorem}

\subsection{Heptagons}

The calculation for the regular heptagon starts out in the same manner as the calculation presented above for regular pentagons.
However, it will turn out that the linear program equivalent to \ref{eq:lin-pent} is in this case not minimized
at $z=0$, and so we will need to add an auxiliary objective to the area as allowed for in \ref{thm:aux}.

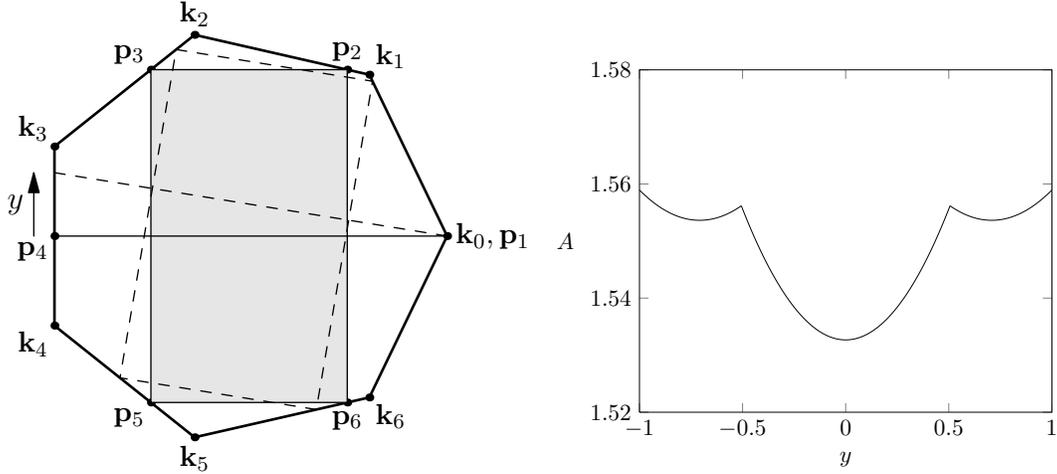
\begin{figure}
\begin{center}
\begin{asy}
import cseblack;
import olympiad;
usepackage("amssymb");
size(200);
pair K0=D("\mathbf{k}_0,\mathbf{p}_1",(1,0),E);
pair K1=D("\mathbf{k}_1",rotate(360./7.)*K0,NE);
pair K2=D("\mathbf{k}_2",rotate(360./7.)*K1,N);
pair K3=D("\mathbf{k}_3",rotate(360./7.)*K2,NW);
pair K4=D("\mathbf{k}_4",rotate(360./7.)*K3,SW);
pair K5=D("\mathbf{k}_5",rotate(360./7.)*K4,S);
pair K6=D("\mathbf{k}_6",rotate(360./7.)*K5,SE);
path KK = D(K0--K1--K2--K3--K4--K5--K6--cycle,linewidth(1));

pair P1=(1., 0.);
pair P2=D("\mathbf{p}_2",(0.516460816041133, 0.806260150051493),N);
pair P3=D("\mathbf{p}_3",(-0.4340236179100764, 0.8062601500514931),NW);
pair P4=D("\mathbf{p}_4",(-0.9009688679024191, 0.),SW);
pair P5=D("\mathbf{p}_5",(-0.4340236179100764, -0.8062601500514931),SW);
pair P6=D("\mathbf{p}_6",(0.516460816041133, -0.806260150051493),S);
path PP = P2--P3--P5--P6--cycle;
fill(PP,lightgray);
D(PP);
D(P1--P4);
pair P4x = P4;

pair P1=(1., 0.);
pair P2=(0.36670303301673796, -0.8404413867659472);
pair P3=(-0.5837814009344715, -0.686832303311513);
pair P4=(-0.9009688679024191, 0.3072181669088686);
pair P5=(-0.31174490092936674, 0.9037741728702919);
pair P6=(0.6387395330218427, 0.7501650894158576);
path PP = P2--P3--P5--P6--cycle;
D(PP,dashed);
D(P1--P4,dashed);

MC ("y",D((P4x-(0.1,0))--(P4-(0.1,0)),Arrow),0.5,W);
\end{asy}
\scalebox{.8}{
    \tikzsetnextfilename{pentplot}
    \begin{tikzpicture}
	\begin{axis}[
	ylabel style={rotate=-90},
		xlabel={$y$}, ylabel={$A$},
		xmin=-1., xmax=1.,
		ymin=1.52, ymax=1.58,
	    ]
	    \addplot[domain=-0.506040792565066:0.506040792565066] {1.5326754446782211 + 0.09176757534725741*x^2};
	    \addplot[domain=0.506040792565066:1] {1.5848482175212668 - 0.08816765628922468*x + 0.06225952189759027*x^2};
	    \addplot[domain=-1:-0.506040792565066] {1.5848482175212668 + 0.08816765628922468*x + 0.06225952189759027*x^2};
	\end{axis}
    \end{tikzpicture}
}
\caption{\label{fig:hept}
Half-length parallelograms in the regular heptagon. The dashed configuration is the local minimum at $y>0$.}
\end{center}
\end{figure}

\begin{table}[t]
    \begin{equation*}
	\begin{aligned}
	    G &= \left(\begin{array}{ccccc}v + 2 u v - 4 u^2 v & \tfrac32 + u - 4 u^2 & 
 11 + 2 u - 16 u^2 & -v - 2 u v + 4 u^2 v & -\tfrac32 - u + 4 u^2 \\ v + 2 u v - 4 u^2 v & 
 \tfrac32 + u - 4 u^2 & -2 + 2 u^2 & -v - 2 u v + 4 u^2 v & -\tfrac32 - u + 
  4 u^2 \\ v + 2 u v - 4 u^2 v & -\tfrac32 - 
  u + 4 u^2 & -11 - 2 u + 16 u^2 & 0 & 0 \\ v + 2 u v - 4 u^2 v & -\tfrac32 - u + 4 u^2 & 
 2 - 2 u^2 & 0 & 0 \\ 0 & 0 & 0 & 2 v - 4 u^2 v & 
 \tfrac12 + 2 u - 2 u^2 \\ 0 & 0 & 0 & 2 v - 4 u^2 v & 
 \tfrac12 + 2 u - 2 u^2 \\ 0 & 0 & 0 & 0 & 0 \\ 0 & 0 & 0 & 0 & 0 \\ -2 v & 0 & 0 & 0 & 0 \end{array}\right.\\
      \eta^T &= \left(\begin{array}{ccccc}\tfrac12, & -2 + 4 u^2, & \tfrac12, & -2 + 
 4 u^2, & \tfrac1{71} (70 + 45 u - 66 u^2), \end{array}\right.\\
	    c^T &= \left(\begin{array}{ccccccccc}7 v + 2 u v - 20 u^2 v, & 0, & 0, & 0, & 1 + u, & 0, & 0, & -1 - u, & 0\end{array}\right)\\
	    z_0^T &= \left(\begin{array}{ccccccccc}0, & 2 - 8 u + 8 u^2, & 0, & -4 u v + 8 u^2 v, & 1, & 0, & 4 u v - 8 u^2 v, & 1, & 0\end{array}\right)
	\end{aligned}
  \end{equation*}
  \caption{\label{tab:heptGc}Constraint Jacobian and objective gradient for the heptagon
  packing problem (continued on opposite page).}
\end{table}
\addtocounter{table}{-1}
\begin{table}[t]
    \begin{equation*}
	\begin{aligned}
	    &\left.\begin{array}{cccc}-2 + 
  2 u^2 & 0 & 0 & 0 \\11 + 2 u - 16 u^2 & 0 & 0 & 0 \\0 & -v - 2 u v + 4 u^2 v & 
 \tfrac32 + u - 4 u^2 & 
 2 - 2 u^2 \\0 & -v - 2 u v + 4 u^2 v & 
 \tfrac32 + u - 4 u^2 & -11 - 2 u + 16 u^2 \\2 - 2 u^2 & 0 & 0 & 0 \\-\tfrac{19}2 - u + 12 u^2 & 0 & 0 & 0 \\0 &
  2 v - 4 u^2 v & -\tfrac12 - 2 u + 2 u^2 & -2 + 2 u^2 \\0 &
  2 v - 4 u^2 v & -\tfrac12 - 2 u + 2 u^2 & 
  \tfrac{19}2 + u - 12 u^2 \\0 & 0 & 0 & 0\end{array}\right)\\
  &\left.\begin{array}{cccc}\tfrac1{71} (-141 - 45 u + 
   208 u^2), & \tfrac1{71} (70 + 45 u - 66 u^2), & \tfrac1{71} (-141 - 45 u + 
   208 u^2), & -5 - 2 u + 12 u^2\end{array}\right)
	\end{aligned}
  \end{equation*}
  \caption{cont.}
\end{table}

We fix a regular heptagon $K=\mathrm{conv} \{\mathbf{k}_i:i=0,\ldots 6\}$, where $\mathbf{k}_i = \mathrm{Rot}_{2\pi i/7} (1,0)$.
In this case, our calculations are performed in the extension field $\mathbb{Q}(u,v)$, where $u=\cos \pi/7$ and $v=\sin \pi/7$.
A minimum-area half-length parallelogram corresponds to the affine diameter $\mathbf{p}_1\mathbf{p}_4$, where
$\mathbf{p}_1 = \mathbf{k}_0$ and $\mathbf{p}_4 = \tfrac12 (\mathbf{k}_3+\mathbf{k}_4)$.
The vertices of the parallelogram are given by
$\mathbf{p}_2=(1-a)\mathbf k_1+a\mathbf k_2$,
$\mathbf{p}_3=(1-b)\mathbf k_2+b\mathbf k_3$,
$\mathbf{p}_5=b\mathbf k_4+(1-b)\mathbf k_5$, and
$\mathbf{p}_6=a\mathbf k_5+(1-a)\mathbf k_6$.
This half-length parallelogram is illustrated in \ref{fig:hept}.

Following the same steps as in the previous sections, we
obtain the following optimization problem:
\begin{equation}\label{eq:opt-hept}
    \begin{aligned}
	\text{minimize } & f(z)=\alpha(\xi'_0(0),\xi'_1(0),\xi'_2(0))+\alpha(\xi'_0(0),\xi'_6(0),\xi'_1(0) )\text,\\
	\text{subj.\ to } & g_1(z)=\alpha(\xi'_0(\mathbf{k}_2),\xi'_0(\mathbf{k}_1),\xi'_2(\mathbf{k}_1))\ge0\\
	                & g_2(z)=\alpha(\xi'_2(\mathbf{k}_2),\xi'_2(\mathbf{k}_1),\xi'_0(\mathbf{k}_1))\ge0\\
	                & g_3(z)=\alpha(\xi'_0(\mathbf{k}_6),\xi'_0(\mathbf{k}_5),\xi'_6(\mathbf{k}_6))\ge0\\
	                & g_4(z)=\alpha(\xi'_6(\mathbf{k}_6),\xi'_6(\mathbf{k}_5),\xi'_0(\mathbf{k}_6))\ge0\\
	                & g_5(z)=\alpha(\xi'_1(\mathbf{k}_3),\xi'_1(\mathbf{k}_2),\xi'_2(\mathbf{k}_3))\ge0\\
	                & g_6(z)=\alpha(\xi'_2(\mathbf{k}_3),\xi'_2(\mathbf{k}_2),\xi'_1(\mathbf{k}_3))\ge0\\
	                & g_7(z)=\alpha(\xi'_1(\mathbf{k}_5),\xi'_1(\mathbf{k}_4),\xi'_6(\mathbf{k}_4))\ge0\\
	                & g_8(z)=\alpha(\xi'_6(\mathbf{k}_5),\xi'_6(\mathbf{k}_4),\xi'_1(\mathbf{k}_4))\ge0\\
	                & g_9(z)=\alpha(\xi'_1(\mathbf{k}_4),\xi'_1(\mathbf{k}_3),\xi'_0(\mathbf{k}_0))\ge0\text.
    \end{aligned}
\end{equation}
The linearization of \ref{eq:opt-hept} around $z=0$ gives
\begin{equation}\label{eq:lin-hept}
    \text{minimize } c\cdot z \text{ subject to } G z\ge 0\text.
\end{equation}
The values of $G$ and $c$ are given in \ref{tab:heptGc}.
Unfortunately, \ref{eq:lin-hept} is unbounded. This can be shown by producing some
$z_u$ such that $c\cdot z_u<0$ and $G z_u\ge0$. In the dual setting, this implies
that there is no $\eta$ such that $c=\eta^T G$ and $\eta>0$.

Due to \ref{thm:aux}, we are allowed to modify the cost function $f(z)$ by adding
auxiliary functions as long as they are negligible in the aggregate
and isometry-invariant. In order for the new problem to be locally minimized, we
need the new gradient $c'$ to lie in the cone $\{\eta^T G:\eta>0\}$. We will take
the following simple form for our modified problem
\begin{equation}\label{eq:mod-hept}
    \begin{aligned}
	\text{minimize } & f'(z) = f(z) +\sum_{r=1}^9 \mu_r g_r(z)\text,\\
	\text{subj.\ to } & \text{same constraints as in \ref{eq:opt-hept}.}
    \end{aligned}
\end{equation}
For a cell $\mathcal P_\xi$ other than the primitive cell $\mathcal P_{\xi_0}$, the modified version is the same
as \ref{eq:mod-hept}, except we replace $\xi'_i$ everywhere with $\xi\circ\xi'_i$.
Since the problem is invariant under common isometry, the problem is equivalent for all the cells
and it is enough to show that \ref{eq:mod-hept} is locally minimized.
Note that $\xi_2\circ\xi_0=\xi_2$ and $\xi_2\circ\xi_2=\xi_0$, so $g_1^{\mathcal P_{\xi_0}}(z) = g_2^{\mathcal P_{\xi_2}}(z)$
and $g_2^{\mathcal P_{\xi_0}}(z) = g_1^{\mathcal P_{\xi_2}}(z)$. Therefore, if $\mu_1=-\mu_2$, the auxiliary addition
$\mu_1 g_1(z)+\mu_2 g_2(z)$ is equal in magnitude and opposite in sign for the two cells $\mathcal P_{\xi_0}$ and
$\mathcal P_{\xi_2}$ sharing the edge $\xi_0(0)\xi_2(0)$. Similar cancellation occurs across the three other
edges if $\mu_3=-\mu_4$, $\mu_5=-\mu_6$, and $\mu_7=-\mu_8$.
The term $\mu_9 g_9(z)$ does not cancel with any neighboring cells, and so we set $\mu_9=0$.
Therefore, when $\mu_i=-\mu_{i+1}$ for $i=1,3,5,7$ and $\mu_9=0$,
the sum of the auxiliary objectives of a collection of cells only has contributions from unmatched edges at the boundary
of the collection and the auxiliary objectives are negligible in the aggregate.
A choice for $\mu_r$ satisfying these conditions and such that $c'$ lies in the cone
$\{\eta^T G:\eta>0\}$ exists if and only if there is some $\eta$ such that
$c=\eta^T G$ and $\eta_1+\eta_2>0$, $\eta_3+\eta_4>0$, $\eta_5+\eta_6>0$, $\eta_7+\eta_8>0$,
and $\eta_9>0$. We can show directly that such $\eta$ exists, and we give an example
in \ref{tab:heptGc}.

We now have that $G z\ge 0$ and $c'\cdot z\le0$
if and only if $G z = 0$ and $c'\cdot z=0$. The rank of $G$ is again $8$,
and so the program \ref{eq:mod-hept} is minimized exactly at the one-dimensional null space of $G$,
which is generated by the vector $z_0$ given in \ref{tab:heptGc}.
We can verify directly
that $f(tz_0)$ is a quadratic function of $t$ minimized at $t=0$, and that
$g_r(tz_0)=0$ identically for $r=1,\ldots,9$. 

Therefore, \ref{eq:mod-hept} satisfies all the conditions of \ref{thm:slice}, and we have:

\begin{theorem}
    The optimal double-lattice packing of regular heptagons, illustrated in \ref{fig:hept}, is strongly extreme.
\end{theorem}

\subsection{General polygons}\label{sec:genpoly}

The structure of the solution in the cases of pentagons and heptagons suggests
that it might be possible to extend the result to general convex polygons.
We will consider a general convex polygon and a double-lattice packing
generated by a half-length parallelogram. We will assume that this half-length
parallelogram is an isolated local minimum of the area among half-length
parallelograms. Due to \ref{thm:nopivot} and \ref{thm:novtx}, this assumption implies
that, except in the exceptional cases, all the contacts in the double-lattice packing fall into the
vertex-to-edge and the edge-to-edge types. Moreover, when the affine diameter
is not unique, the contact between $K$ and $\mathrm{Tran}_{\mathbf{p}_1-\mathbf{p}_4}(K)$ is
edge-to-edge, but the associated constraint can be relaxed to a vertex-to-edge
constraint so that our analysis may follow the generic case.
We find that the only data about
the polygon that enters into the calculation are the following:
\begin{enumerate}
    \item The coordinates of vertices of the minimum-area half-length parallelogram
	and the endpoints of a corresponding affine diameter, with at least one endpoint
	in the interior of an edge. Without loss of
	generality, we may assume the affine diameter is horizontal, of length 2,
	and bisected by the origin. The remaining data are encoded into the following
	parameters: $h$ is the height of parallelogram and $a$ is the horizontal shift
	between the top and bottom sides of the parallelogram. The position of the parallelogram
	center, not determined by these parameters, is eliminated during the calculation.
    \item The angle of the polygon edges on which the vertices of the parallelogram and
	the endpoint of the affine diameter lie, $\phi_i$, $i=2,3,4,5,6$.
    \item The distance and direction along the edge from those points to the nearest polygon vertex, that 
	is, half the length of the contact. We denote these distances $l_i$, $i=2,3,4,5,6$. For the directions,
	we will assume the nearest polygon vertex is counterclockwise from the parallelogram vertex, but we will argue later that
	this assumption has no effect on the subsequent analysis.
\end{enumerate}

The assumption that the area of the half-length parallelogram is minimized can be written as
\begin{equation}\label{eq:stat}
    a= h \cot\phi_4 - 
    \frac{\sin\phi_3 \sin(\phi_4 - \phi_2)}{\sin\phi_4 \sin(\phi_3-\phi_2)} + 
    \frac{\sin\phi_5 \sin(\phi_6 - \phi_4)}{\sin\phi_4 \sin(\phi_6-\phi_5)}
\end{equation}

As in the previous sections the objective is given by 
the area of the quadrilateral $\xi'_0(0)\xi'_6(0)\xi'_1(0)\xi'_2(0)$, and we parameterize
the search space using $z=(x_0,\ldots,\theta_6)\in\mathbb{R}^9$ and $\xi'_1=\xi_1$.
The oriented triangles to be used to represent the nonoverlap constraints depend on the directions
of $l_i$, $i=2,3,5,6$. With the contact directions assumed, the optimization problem is:
\begin{equation}\label{eq:opt-gen}
    \begin{aligned}
	\text{minimize } & f(z)=\alpha(\xi'_0(0),\xi'_1(0),\xi'_2(0))-\alpha(\xi'_0(0),\xi'_6(0),\xi'_1(0) )\text,\\
	\text{subj.\ to } & g_1(z)=\alpha(\xi'_0(\mathbf{p}_2+l_2\mathbf{u}_2),\xi'_0(\mathbf{p}_2-l_2\mathbf{u}_2),\xi'_2(\mathbf{p}_2+l_2\mathbf{u}_2))/l_2^2\ge0\\
	                  & g_2(z)=\alpha(\xi'_2(\mathbf{p}_2+l_2\mathbf{u}_2),\xi'_2(\mathbf{p}_2-l_2\mathbf{u}_2),\xi'_0(\mathbf{p}_2+l_2\mathbf{u}_2))/l_2^2\ge0\\
	                  & g_3(z)=\alpha(\xi'_2(\mathbf{p}_3+l_3\mathbf{u}_3),\xi'_2(\mathbf{p}_3-l_3\mathbf{u}_3),\xi'_1(\mathbf{p}_3+l_3\mathbf{u}_3))/l_3^2\ge0\\
	                  & g_4(z)=\alpha(\xi'_1(\mathbf{p}_3+l_3\mathbf{u}_3),\xi'_1(\mathbf{p}_3-l_3\mathbf{u}_3),\xi'_2(\mathbf{p}_3+l_3\mathbf{u}_3))/l_3^2\ge0\\
	                  & g_5(z)=\alpha(\xi'_1(\mathbf{p}_5+l_5\mathbf{u}_5),\xi'_1(\mathbf{p}_5-l_5\mathbf{u}_5),\xi'_6(\mathbf{p}_5+l_5\mathbf{u}_5))/l_5^2\ge0\\
	                  & g_6(z)=\alpha(\xi'_6(\mathbf{p}_5+l_5\mathbf{u}_5),\xi'_6(\mathbf{p}_5-l_5\mathbf{u}_5),\xi'_1(\mathbf{p}_5+l_5\mathbf{u}_5))/l_5^2\ge0\\
	                  & g_7(z)=\alpha(\xi'_6(\mathbf{p}_6+l_6\mathbf{u}_6),\xi'_6(\mathbf{p}_6-l_6\mathbf{u}_6),\xi'_0(\mathbf{p}_6+l_6\mathbf{u}_6))/l_6^2\ge0\\
	                  & g_8(z)=\alpha(\xi'_0(\mathbf{p}_6+l_6\mathbf{u}_6),\xi'_0(\mathbf{p}_6-l_6\mathbf{u}_6),\xi'_6(\mathbf{p}_6+l_6\mathbf{u}_6))/l_6^2\ge0\\
	                  & g_9(z)=\alpha(\xi'_1(\mathbf{p}_4+l_4\mathbf{u}_4),\xi'_1(\mathbf{p}_4-l_4\mathbf{u}_4),\xi'_0(\mathbf{p}_1))/l_4^2\ge0\text.
    \end{aligned}
\end{equation}

We linearize the problem to obtain a problem of the form \ref{eq:lin-hept}.
The constraint matrix $G$ is singular, and we can obtain right and left null
space vectors $z_0$ and $\eta_0$, whose values are given in \ref{tab:genz0e0}.
In fact, the null spaces are one-dimensional, as can be seen by calculating 
$\det(G-\eta_0 z^T)/(z_0\cdot z_0) = 2^{14}\sin(\phi_3-\phi2)\sin(\phi_6-\phi_5)/(l_2 l_3 l_4 l_5 l_6) \neq0$.
Note that the variables associated with the rotations $\theta_i$ in $z_0$ are zero,
and the assignment $z=tz_0$ corresponds to evolving the half-length parallelogram.
We therefore will have that $f(tz_0)$ has an isolated local minimum at $t=0$ and that $g_r(tz_0)=0$ for all $t$.
We note also that $z_0\cdot c = 0$ and so $c$ is contained in the row space of $G$.
We solve for the vector $\eta$ such that $\eta\cdot G = c$ and $\eta\cdot\eta_0=0$. We obtain the
following values:
\begin{equation}\label{eq:etas}
    \begin{aligned}
	\eta_1+\eta_2 &= -l_2\sin\phi_3/\sin(\phi_3-\phi_2)\\
	\eta_3+\eta_4 &= l_3\sin\phi_2/\sin(\phi_3-\phi_2)\\
	\eta_5+\eta_6 &= l_5\sin\phi_6/\sin(\phi_6-\phi_5)\\
	\eta_7+\eta_8 &= -l_6\sin\phi_5/\sin(\phi_6-\phi_5)\\
	\eta_9 &= -\frac{l_4}{\sin\phi_4}\left(h - \frac{1}{\cot\phi_3-\cot\phi_2} - \frac{1}{\cot\phi_6-\cot\phi_5}\right)\text.
    \end{aligned}
\end{equation}
Since all the above are positive, we can proceed as in the case of heptagons to include auxiliary functions
that would make all the $\eta_i$'s individually positive and be negligible in the
aggregate and isometry-invariant.
Therefore, we have shown that the packing is strongly extreme.

\begin{table}
    \begin{equation*}
	\begin{aligned}
	    z_0= &\left(\begin{array}{ccccc}
		\cos\phi_4, & \sin\phi_4, & 0, & \cos(\phi_3) \sin(\phi_2 - \phi_4)/\sin(\phi_2 - \phi_3), & \sin(\phi_3) \sin(\phi_2 - \phi_4) / \sin(\phi_2 - \phi_3),\end{array}\right.\\
	    &\left.\begin{array}{cccc} 0, & 
		\cos(\phi_5) \sin(\phi_4 - \phi_6) / \sin(\phi_5 - \phi_6), & \sin(\phi_5) \sin(\phi_4 - \phi_6) / \sin(\phi_5 - \phi_6), & 0 \end{array}\right)\\
	    \eta_0 = &\left(\begin{array}{ccccccccc} 1 & -1 & 1 & -1 & 1 & -1 & 1 & -1 & 0\end{array}\right)
	\end{aligned}
    \end{equation*}
    \caption{\label{tab:genz0e0}Right and left null space generators of the constraint matrix in the general polygon packing problem.}
\end{table}

To see that the directions of the contacts do not matter, simply relax the 
correct signed-area constraints to the signed-area constraints derived from a
proper subinterval of the contact with the contact direction as assumed in the proof.
This change is in fact a relaxation in an appropriate neighborhood. Since
the relaxed problem is still solved uniquely at $\xi_i'=\xi_i$, the same
is true of the original. The same argument also takes care of contacts
at full edges.

We have therefore proved that any double-lattice packing that is an isolated
local minimum among double-lattice packings and is not one of the exceptional
types of \ref{def:except} is also strongly extreme.
In practice, we have a procedure which, when combined with the Mount algorithm, takes a general polygon as
input, produces the densest double lattice packing, and certifies it as local maximum among general packings.
If the densest double lattice packing is not an isolated maximum or is of exceptional type, further analysis is necessary.

\textbf{Acknowledgments}.
Y.\ K.\ was supported by an Omidyar Fellowship at the Santa Fe Institute. W.\ K.\ was supported by Austrian Science Fund (FWF) Project 5503 and National Science Foundation (NSF) Grant No. 1104102.  We wish to thank the Erwin Schr\"odinger International Institute for Mathematical Physics (ESI) and the Institute for Computational and Experimental Research (ICERM) for supporting our visits and hosting programs that facilitated this work.

\bibliography{ngon}

\newcommand{\etalchar}[1]{$^{#1}$}
\begin{thebibliography}{dGvRD11}

\bibitem[BK13]{bezdek2013dense}
{Andr\'as} Bezdek and {W\l{}odzimierz} Kuperberg.
\newblock Dense packing of space with various convex solids.
\newblock In Imre {B\'ar\'any}, {K\'aroly}~J. {B\"or\"oczky}, {G\'abor} {Fejes
  T\'oth}, and {J\'anos} Pach, editors, {\em Geometry -- Intuitive, Discrete,
  and Convex}, pages 65--90. Springer, Berlin, 2013.

\bibitem[BMP05]{brass2005research}
Peter Brass, William O.~K. Moser, and {J\'anos} Pach.
\newblock {\em Research Problems in Discrete Geometry}.
\newblock Springer, 2005.

\bibitem[BSS06]{bazaraa2006nonlinear}
Mokhtar~S. Bazaraa, Hanif~D. Sherali, and C.~M. Shetty.
\newblock {\em Nonlinear Programming}.
\newblock Wiley, 2006.

\bibitem[CE03]{cohn2003new}
Henry Cohn and Noam~D. Elkies.
\newblock New upper bounds on sphere packings {I}.
\newblock {\em Annals of Mathematics}, 157(2):689--714, 2003.

\bibitem[CGSS99]{conway1999recent}
John~H. Conway, Chaim Goodman-Strauss, and Neil J.~A. Sloane.
\newblock Recent progress in sphere packing.
\newblock {\em Current Developments in Mathematics}, 1999:37--76, 1999.

\bibitem[dGvRD11]{Graaf2011}
Joost de~Graaf, René van Roij, and Marjolein Dijkstra.
\newblock Dense regular packings of irregular nonconvex particles.
\newblock {\em Physical Review Letters}, 107(15):155501, 2011.

\bibitem[dOV13]{mario2013computing}
Fernando~{M\'ario} de~Oliveira, Filho and Frank Vallentin.
\newblock Computing upper bounds for the packing density of congruent copies of
  a convex body {I}.
\newblock {\em arXiv:1308.4893}, 2013.

\bibitem[FT50]{fejes1950some}
{L\'aszl\'o} Fejes~{T\'oth}.
\newblock Some packing and covering theorems.
\newblock {\em Acta Scientiarum Mathematicarum (Szeged)}, 12(A):62--67, 1950.

\bibitem[GEK11]{gravel2011upper}
Simon Gravel, Veit Elser, and Yoav Kallus.
\newblock Upper bound on the packing density of regular tetrahedra and
  octahedra.
\newblock {\em Discrete \& Computational Geometry}, 46(4):799--818, 2011.

\bibitem[Gro63]{groemer1963existenzsatze}
Helmut Groemer.
\newblock {Existenzs\"atze} {f\"ur} {Lagerungen} im {Euklidischen} {Raum}.
\newblock {\em Mathematische Zeitschrift}, 81(3):260--278, 1963.

\bibitem[HAB{\etalchar{+}}15]{hales2015formal}
Thomas Hales, Mark Adams, Gertrud Bauer, Dat~Tat Dang, John Harrison, Truong~Le
  Hoang, Cezary Kaliszyk, Victor Magron, Sean McLaughlin, Thang~Tat Nguyen,
  Truong~Quang Nguyen, Tobias Nipkow, Steven Obua, Joseph Pleso, Jason Rute,
  Alexey Solovyev, An~Hoai~Thi Ta, Trung~Nam Tran, Diep~Thi Trieu, Josef Urban,
  Ky~Khac Vu, and Roland Zumkeller.
\newblock A formal proof of the {Kepler} conjecture.
\newblock {\em arXiv:1501.02155}, 2015.

\bibitem[Hal12]{hales2012blueprint}
Thomas Hales.
\newblock {\em Dense Sphere Packings: A Blueprint for Formal Proofs}.
\newblock Cambridge University Press, 2012.

\bibitem[Kal14]{kallus2014ball}
Yoav Kallus.
\newblock The 3-ball is a local pessimum for packing.
\newblock {\em Advances in Mathematics}, 264:355--370, 2014.

\bibitem[Kal15]{kallus2015pessimal}
Yoav Kallus.
\newblock Pessimal packing shapes.
\newblock {\em Geometry \& Topology}, 19(1):343--363, 2015.

\bibitem[KK90]{kuperberg1990double}
Greg Kuperberg and {W\l{}odzimierz} Kuperberg.
\newblock Double-lattice packings of convex bodies in the plane.
\newblock {\em Discrete \& Computational Geometry}, 5(1):389--397, 1990.

\bibitem[Kup00]{Kuperberg2000notions}
Greg Kuperberg.
\newblock Notions of denseness.
\newblock {\em Geometry \& Topology}, 4:277–292, 2000.

\bibitem[Mar03]{martinet2003perfect}
Jacques Martinet.
\newblock {\em Perfect lattices in Euclidean spaces}.
\newblock Springer Science \& Business Media, 2003.

\bibitem[Mou91]{mount1991densest}
David~M. Mount.
\newblock The densest double-lattice packing of a convex polygon.
\newblock {\em Discrete \& Computational Geometry: Papers from the DIMACS
  Special Year}, 6:245--262, 1991.

\bibitem[Rei33]{reinhardt1933dichteste}
Karl Reinhardt.
\newblock {\"Uber} die dichteste {gitterf\"ormige} {Lagerung} kongruenter
  {Bereiche} in der {Ebene} und eine besondere {Art} konvexer {Kurven}.
\newblock In {\em Abh. aus dem Math. Seminar der Hamburgische Univ}, volume~9,
  pages 216--230, 1933.

\bibitem[Sch13]{schurmann2013strict}
Achill Sch{\"u}rmann.
\newblock Strict periodic extreme lattices.
\newblock {\em Diophantine Methods, Lattices, and Arithmetic Theory of
  Quadratic Forms}, 587:185, 2013.

\bibitem[Van11]{vance2011improved}
Stephanie Vance.
\newblock Improved sphere packing lower bounds from {Hurwitz} lattices.
\newblock {\em Advances in Mathematics}, 227(5):2144--2156, 2011.

\bibitem[Ven13]{venkatesh2013note}
Akshay Venkatesh.
\newblock A note on sphere packings in high dimensions.
\newblock {\em International Mathematics Research Notices}, 2013(7):1628, 2013.

\end{thebibliography}
\end{document}